\documentclass[12pt]{amsart}

\usepackage{amsfonts, amsthm, amsmath, amssymb}

\usepackage{makecell,multirow,diagbox}

\usepackage{amscd}

\usepackage[latin2]{inputenc}

\usepackage{t1enc}

\usepackage[mathscr]{eucal}

\usepackage{indentfirst}

\usepackage{graphicx}

\usepackage{graphics}

\usepackage{pict2e}

\usepackage{mathrsfs}

\usepackage{framed}

\usepackage{mathabx}
\usepackage{booktabs}
\usepackage{makecell}
\usepackage{enumerate}
\usepackage[pagebackref]{hyperref}
\hypersetup{colorlinks=true}
\usepackage{cite}
\usepackage{color}
\usepackage{epic}
\usepackage{hyperref} 
\numberwithin{equation}{section}
\theoremstyle{plain}

\usepackage{tikz}
\usetikzlibrary{patterns}

\usetikzlibrary{calc}
\usetikzlibrary{shapes.geometric}

\tikzset{
  c/.style={every coordinate/.try}
}

\usetikzlibrary{arrows,shapes,positioning}
\usetikzlibrary{decorations.markings}
\tikzstyle arrowstyle=[scale=1]
\tikzstyle directed=[postaction={decorate,decoration={markings,mark=at position 0.6 with {\arrow[arrowstyle]{stealth};}}}]
\tikzstyle reverse directed=[postaction={decorate,decoration={markings,mark=at position 0.4 with {\arrowreversed[arrowstyle]{stealth};}}}]
\tikzstyle dot=[style={circle,inner sep=1pt,fill}]

\newtheorem{theorem}{Theorem}[section]

\newtheorem{lemma}[theorem]{Lemma}

\newtheorem{corollary}[theorem]{Corollary}

\newtheorem{St}[theorem]{Statement}

\newtheorem{proposition}[theorem]{Proposition}

\newtheorem{Obser}[theorem]{Observation}

\theoremstyle{definition}

\newtheorem{example}[theorem]{Example}

\newtheorem{conj}[theorem]{Conjecture}

\newtheorem{remark}[theorem]{Remark}

\newtheorem{?}[theorem]{Problem}

\def\Ides{\mathrm{Ides}}
\def\ides{\mathrm{ides}}
\def\id{\mathrm{id}}
\def\Vlrmax{\mathrm{Lrmax}}
\def\Vlrmin{\mathrm{Lrmin}}
\def\Vrlmax{\mathrm{Rlmax}}
\def\Des{\mathrm{Des}}
\def\St{\mathrm{St}}
\def\st{\mathrm{st}}
\def\SS{\mathfrak{S}}
\def\lrmax{\mathrm{lrmax}}
\def\lrmin{\mathrm{lrmin}}
\def\rlmax{\mathrm{rlmax}}
\def\rlmin{\mathrm{rlmin}}

\def\iar{\mathrm{iar}}
\def\Iar{\mathrm{Iar}}
\def\W{\mathcal{W}}

\def\satu{\mathrm{satu}}
\def\exc{\mathrm{exc}}
\def\lmaxz{\mathrm{lmaxz}}
\def\dist{\mathrm{dist}}
\def\zero{\mathrm{zero}}
\def\satu{\mathrm{satu}}
\def\rep{\mathrm{rep}}
\def\iasc{\mathrm{iasc}}
\def\st{\mathrm{st}}
\def\St{\mathrm{St}}

\def\asc{\mathrm{asc}}

\def\Des{\mathrm{Des}}
\def\des{\mathrm{des}}
\def\max{\mathrm{max}}

\def\Pk{\mathrm{Pk}}

\def\I{{\bf I}}
\def\N{\mathbb{N}}
\def\Br{\mathrm{Br}}
\def\br{\mathrm{br}}
\def\B{\mathfrak{B}}
\def\s{{\bf s}}

\def\boxit#1{\leavevmode\hbox{\vrule\vtop{\vbox{\kern.33333pt\hrule
    \kern1pt\hbox{\kern1pt\vbox{#1}\kern1pt}}\kern1pt\hrule}\vrule}}

\usepackage{collectbox}



\begin{document}

\title[A bijection for  patterns of length $5$]{A bijection for  length-$5$ patterns  in permutations}

\author[J.N. Chen]{Joanna N. Chen}
\address[Joanna N. Chen]{College of Science, Tianjin University of Technology, Tianjin 300384, P.R. China}
\email{joannachen@tjut.edu.cn}

\author[Z. Lin]{Zhicong Lin}
\address[Zhicong Lin]{Research Center for Mathematics and Interdisciplinary Sciences, Shandong University, Qingdao 266237, P.R. China}
\email{linz@sdu.edu.cn}

\date{\today}
\begin{abstract}
A bijection between $(31245,32145,31254,32154)$-avoiding permutations  and $(31425,32415,31524,32514)$-avoiding permutations is constructed, which preserves five classical set-valued statistics. Combining with two codings of permutations due respectively to Baril--Vajnovszki and Martinez--Savage proves an enumerative conjecture posed by Gao and Kitaev. Moreover, the generating function for the common counting sequence is proved to be algebraic.
\end{abstract}
\keywords{}
\maketitle
%

\section{Introduction}\label{sec:intro}
Given  two words $P=p_1 p_2 \cdots p_k$ and $W=w_1w_2\ldots w_n$ over $\N$, where $k\leq n$, we say
 that $W$ contains the \emph{pattern} $P$  if there exists integers
 $i_1 < i_2 < \cdots < i_k$ such that
 $w_{i_1}w_{i_2} \cdots w_{i_k}$  is order isomorphic to $P$. Otherwise, we say that $W$  \emph{avoids} $P$, or  $W$ is \emph{$P$-avoiding}. For a set of words $\mathcal{W}$, the set of words in $\mathcal{W}$ avoiding patterns $P_1,\ldots,P_r$ is denoted by $\mathcal{W}(P_1,\ldots,P_r)$.


Let $\SS_n$ be the set of all
permutations of $[n]:=\{1,2,\cdots,n\}$. Permutations are viewed as words and the study of patterns in permutations and words from the enumerative aspect can be traced back to the work of MacMahon~\cite{Mac1915}. For over a half century, this theme of research  has been the focus  in enumerative and bijective combinatorics (see Kitaev's monograph~\cite{Kit2011}).
This paper is motivated by an enumerative conjecture posed by Gao and Kitaev~\cite[Table~5]{gk}, which in the language of pattern avoidance asserts that
\begin{conj}[Gao--Kitaev]
For $n\geq1$,
$$
|\SS_n(45312,45321,54312,54321)|=|\SS_n(31245,32145,31254,32154)|.
$$
\end{conj}

The main objective of this paper is to construct a bijection preserving  five  set-valued statistics between $\SS_n(31245,32145,31254,32154)$ and $\SS_n(31425,32415,31524,32514)$, which together with the works in~\cite{ms} and~\cite{CL} implies a refinement of Gao--Kitaev's conjecture (see Proposition~\ref{refine:GK}). In order to state our main result, we still need some further notations and definitions.

Given  $\pi=\pi_1 \pi_2 \cdots \pi_n \in \SS_n$,
we say that $i$ is a \emph{descent} of $\pi$ if $\pi_i> \pi_{i+1}$.
Define the \emph{descent set} $\Des(\pi)$  to be the set of all descents of $\pi$
and $\Ides(\pi)$ to be the \emph{inverse descent set} of $\pi$ as
\begin{equation*}
  \Ides(\pi)=\{i \in [n-1] \colon \pi^{-1}(i)>\pi^{-1}(i+1)\}.
\end{equation*}
If $\pi$ has $k$ descents, then $\pi$ is the union of $k+1$ maximal increasing subsequences of consecutive entries.  These
are called the \emph{ascending runs} of $\pi$.
We denote by $\Iar(\pi)$  the set of elements in the  {\em initial ascending run} of $\pi$.
A letter $\pi_i$ is said to be a \emph{left-to-right maximum} of $\pi$ if
$\pi_i> \pi_j$ for all $j<i$. Denote $\Vlrmax(\pi)$ the set of
all  left-to-right maxima of $\pi$. Similarly,
we may define  the \emph{left-to-right minimum} and the \emph{right-to-left maximum} of $\pi$ and let $\Vlrmin(\pi)$ and $\Vrlmax(\pi)$  be the corresponding sets.

Throughout the paper, we make the convention that if ``$\St$'' is
a set-valued statistic, then ``$\st$'' is the corresponding numerical statistic.
For example, $\des(\pi)$ denotes the number of descents of $\pi$. Note that ``$\des$'' and ``$\ides$'' are known as {\em Eulerian} statistics over permutations, while ``$\lrmax$'', ``$\lrmin$'', ``$\rlmax$'' are {\em Stirling} statistics. These classical statistics have been investigated  over several pattern avoidance classes of permutations~\cite{blo,CN,kl}.

\begin{theorem}\label{thm:fiveplus}
There exists a bijection $$\alpha: \SS_n(31245,32145,31254,32154)\rightarrow\SS_n(31425,32415,31524,32514)$$ that preserves the quintuple of set-valued statistics $(\Ides,\Vlrmax,\Vlrmin,\Vrlmax,\Iar)$.
\end{theorem}


Notice that all statistics above can also be defined on words of distinct  letters in the sense of isomorphism
and we denote $\W_n$ as the set of such words with length $n$. Our bijection $\alpha$ will be constructed in some recursive way over $\W_n$, rather than over $\SS_n$ directly. Its construction is based on a new bijection between $\SS_n(3124,3214)$ and $\SS_n(3142,3241)$ via block decompositions.

The class of $(45312,45321,54312,54321)$-avoiding permutations arose from Kitaev and Remmel's study of quadrant marked mesh patterns~\cite{kr}. Its enumeration sequence has been registered as A212198 in the OEIS~\cite{oeis}. However, the problem to compute the generating function for this integer sequence remains open. We will solve this problem by studying $(201,210)$-avoiding inversion sequences (see Section~\ref{sec:4}) which were known~\cite{ms} to be in bijection with $(45312,45321,54312,54321)$-avoiding permutations.

The rest of this paper is organized as follows. In Section~\ref{sec:2}, we construct a new bijection between $\SS_n(3124,3214)$ and $\SS_n(3142,3241)$ via block decompositions, which is used in Section~\ref{sec:3} to built our main bijection $\alpha$. Section~\ref{sec:4} is devoted to the study of $(201,210)$-avoiding inversion sequences, including a refinement of Gao--Kitaev's conjecture, a succession rule and two functional equations.

\section{A new bijection between $\SS_n(3124,3214)$ and $\SS_n(3142,3241)$ via block decompositions}
\label{sec:2}
In~\cite{kl}, a bijection  between $\SS_n(3124,3214)$ and $\SS_n(3142,3241)$ was constructed through the intermediate structure of $021$-avoiding inversion sequences. In this section, a new bijection preserving  more set-valued statistics is constructed, which will be used to define our main bijection $\alpha$ in next section. One interesting feature of this new bijection is that it preserves the number of blocks that we now introduce.

For $\pi\in \SS_n$, let $\pi_0=\pi_{n+1}=0$. Define the set of \emph{ peaks} of $\pi$ by
\begin{align*}
  \Pk(\pi)  =\{\pi_i:  1 \leq i \leq n \,\, \text{and}\,\,  \pi_{i-1}< \pi_i > \pi_{i+1}\}.
\end{align*}
 For  $\pi \in \SS_n(3124,3214)\cup \SS_n(3142,3241)$,
 define the set of representatives coming from each block of $\pi$ by
 \[\Br(\pi)=\Vlrmax(\pi) \cap \Pk(\pi).\]
It turns out later that $\br(\pi)$ equals the number of blocks in our two block decompositions of $\pi$.

\begin{theorem}\label{thm:fourplus}
There exists a bijection $\varphi:\SS_n(3124,3214)\rightarrow\SS_n(3142,3241)$ that preserves the quintuple of set-valued statistics $(\Br,\Ides, \Vlrmax, \Vlrmin,\Iar)$.
\end{theorem}

We begin with the analysis of the structure of $\{3124,3214\}$-avoiding  words of different letters.
Given 
$w=w_1w_2 \cdots w_n \in \W_n(3124,3214)$,  assume that $w=w'w''$ where
 $w'$ is the union of ascending runs
$w_1 w_2 \cdots w_{i_1}$,  $w_{i_1+1}\cdots w_{i_2}$,
$\ldots$, $w_{i_{k-1}+1} \cdots w_{i_{k}}$, $w_{i_{k}+1} \cdots w_x$ and $w''=w_{x+1} \cdots w_n$ with $w_x=\max(w)$.
We have the following proposition (see Fig.~\ref{fig:312214} the visualization).

\begin{proposition}\label{prop:stru1}
Suppose that $w=w_1w_2 \cdots w_n \in \W_n(3124,3214)$  as written in the above version, then
\begin{itemize}
  \item[1.]  elements $w_j$ with $j \leq x$ and $j \neq i_1+1, i_2+1, \ldots, i_k+1$ are left-to-right maxima of $w$, namely,
  $$
  \Vlrmax(w)=\{w_1, \ldots, w_{i_1}, w_{i_1+2},\ldots,w_{i_2}, w_{i_2+2},\ldots, w_{i_k}, w_{i_k+2}, \ldots, w_x\}.
  $$
  \item[2.] $w''=b b_{k} \cdots b_{1}$, where
  $b$ is a block of consecutive elements larger than
    $w_{i_k}$,
      $b_{j}$ is a block of consecutive elements smaller than $w_{ i_j}$
      and larger than $w_{i_{j-1}}$ with $ 1< j <k$, $b_{1}$ is a block of consecutive elements smaller than $w_{i_1}$.
\end{itemize}
\end{proposition}

\begin{proof}
To prove property $1$, it suffices to show that  $w_{i_j+2}>w_{i_j}$ for each $1 \leq j \leq k$. This is obviously true, otherwise $w_{i_j} w_{i_j+1} w_{i_j+2} w_x$ will form a $3124$ or $3214$-pattern.
Moreover, assume that $w_h$ is the rightmost element larger than $w_{i_j}$, clearly $w_h$ is a right-to-left maximum of $w$.
We claim that $w_l>w_{i_j}$ for all $ i_j+1< l < h$, otherwise
$w_{i_j} w_{i_j+1} w_{l} w_h$ will be an instance of pattern $3124$ or $3214$. This leads us to property $2$ and the proof now is completed.
\end{proof}

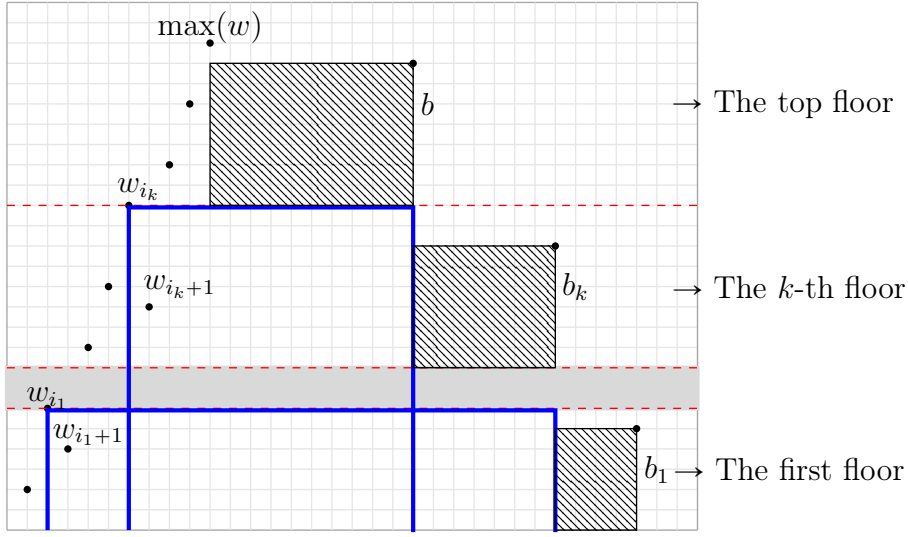
\begin{figure}
\begin{center}
\begin{tikzpicture}[line width=0.5pt,scale=0.27]
\coordinate (O) at (0,0);
\coordinate (O1) at (0,-1);
\coordinate (C) at (1,-25);
\draw [gray!70] (O)--++(34,0);
\draw [gray!20] (O)++(0,-1)--++(34,0);
\draw [gray!20] (O)++(0,-2)--++(34,0);
\draw [gray!20] (O)++(0,-3)--++(34,0);
\draw [gray!20] (O)++(0,-4)--++(34,0);
\draw [gray!20] (O)++(0,-5)--++(34,0);
\draw [gray!20] (O)++(0,-6)--++(34,0);
\draw [gray!20] (O)++(0,-7)--++(34,0);
\draw [gray!20] (O)++(0,-8)--++(34,0);
\draw [gray!20] (O)++(0,-9)--++(34,0);
\draw [gray!20] (O)++(0,-10)--++(34,0);
\draw [gray!20] (O)++(0,-11)--++(34,0);
\draw [gray!20] (O)++(0,-12)--++(34,0);
\draw [gray!20] (O)++(0,-13)--++(34,0);
\draw [gray!20] (O)++(0,-14)--++(34,0);
\draw [gray!20] (O)++(0,-15)--++(34,0);
\draw [gray!20] (O)++(0,-16)--++(34,0);
\draw [gray!20] (O)++(0,-17)--++(34,0);
\draw [gray!20] (O)++(0,-18)--++(34,0);
\draw [gray!20] (O)++(0,-20)--++(34,0);
\draw [gray!20] (O)++(0,-21)--++(34,0);
\draw [gray!20] (O)++(0,-22)--++(34,0);
\draw [gray!20] (O)++(0,-23)--++(34,0);
\draw [gray!20] (O)++(0,-24)--++(34,0);
\draw [gray!20] (O)++(0,-25)--++(34,0);
\draw [gray!70] (O)++(0,-26)--++(34,0);
\draw [gray!70] (O)--++(0,-26);
\draw [gray!20] (O)++(1,0)--++(0,-26);
\draw [gray!20] (O)++(2,0)--++(0,-26);
\draw [gray!20] (O)++(3,0)--++(0,-26);
\draw [gray!20] (O)++(4,0)--++(0,-26);
\draw [gray!20] (O)++(5,0)--++(0,-26);
\draw [gray!20] (O)++(6,0)--++(0,-26);
\draw [gray!20] (O)++(7,0)--++(0,-26);
\draw [gray!20] (O)++(8,0)--++(0,-26);
\draw [gray!20] (O)++(9,0)--++(0,-26);
\draw [gray!20] (O)++(10,0)--++(0,-26);
\draw [gray!20] (O)++(11,0)--++(0,-26);
\draw [gray!20] (O)++(12,0)--++(0,-26);
\draw [gray!20] (O)++(13,0)--++(0,-26);
\draw [gray!20] (O)++(14,0)--++(0,-26);
\draw [gray!20] (O)++(15,0)--++(0,-26);
\draw [gray!20] (O)++(16,0)--++(0,-26);
\draw [gray!20] (O)++(17,0)--++(0,-26);
\draw [gray!20] (O)++(18,0)--++(0,-26);
\draw [gray!20] (O)++(19,0)--++(0,-26);
\draw [gray!20] (O)++(20,0)--++(0,-26);
\draw [gray!20] (O)++(21,0)--++(0,-26);
\draw [gray!20] (O)++(22,0)--++(0,-26);
\draw [gray!20] (O)++(23,0)--++(0,-26);
\draw [gray!20] (O)++(24,0)--++(0,-26);
\draw [gray!20] (O)++(25,0)--++(0,-26);
\draw [gray!20] (O)++(26,0)--++(0,-26);
\draw [gray!20] (O)++(27,0)--++(0,-26);
\draw [gray!20] (O)++(28,0)--++(0,-26);
\draw [gray!20] (O)++(29,0)--++(0,-26);
\draw [gray!20] (O)++(30,0)--++(0,-26);
\draw [gray!20] (O)++(31,0)--++(0,-26);
\draw [gray!20] (O)++(32,0)--++(0,-26);
\draw [gray!20] (O)++(33,0)--++(0,-26);
\draw [gray!20] (O)++(34,0)--++(0,-26);
\draw [gray!70] (O)++(34,0)--++(0,-26);

\draw [fill=gray!30,ultra thick,gray!30] (0,-18) rectangle (34,-20);
\fill[black!100] (C)++(0,1) circle(1ex)++(1,4) circle(1ex)
++(1,-2) circle(1ex)
++(0,2)
++(1,3) circle(1ex)++(1,3) circle(1ex)
++(1,4) circle(1ex)++(1,-5) circle(1ex)++(1,7) circle(1ex)
++(1,3) circle(1ex)++(1,3) circle(1ex)++(10,-1) circle(1ex)
++(7,-9) circle(1ex)
++(4,-9) circle(1ex);
\draw [dashed,red] (C)++(-1,5)--++(34,0);
\draw [dashed,red] (C)++(-1,7)--++(34,0);
\draw [dashed,red] (C)++(-1,15)--++(34,0);

\draw [ultra thick,blue] (O)++(2,-20)--++(0,-6);
\draw [ultra thick,blue] (O)++(2,-20.1)--++(25,0)--++(0,-6);

\draw [ultra thick,blue] (O)++(6,-10)--++(0,-16);
\draw [ultra thick,blue] (O)++(6,-10.1)--++(14,0)--++(0,-16);

\draw[pattern=north west lines] (27,-21) rectangle (31,-26);
\draw[pattern=north west lines] (20,-12) rectangle (27,-18);
\draw[pattern=north west lines] (10,-3) rectangle (20,-10);

\path (2,-19.5)  node {$w_{i_1}$}
++(2,-1.8) node {$w_{i_1+1}$};

\path (6.5,-9.2)  node {$w_{i_k}$}
++(2,-4.8) node {$w_{i_k+1}$};

\path (10,-1.2)  node {$\max{(w)}$};

\path (32,-23)  node {$b_1$};
\path (28,-14)  node {$b_k$};
\path (20.7,-5)  node {$b$};

\path (38.5,-23)  node {$\boldsymbol{\rightarrow}$ The first floor};
\path (38.5,-14)  node {$\boldsymbol{\rightarrow}$ The $k$-th floor};
\path (38.2,-5)  node {$\boldsymbol{\rightarrow}$ The top floor};
\end{tikzpicture}
\caption{The structure of a $\{3124, 3214\}$-avoiding word of different letters.}\label{fig:312214}
\end{center}
\end{figure}

Based on Proposition \ref{prop:stru1}, we see that a
$\{3124,3214\}$-avoiding word of different letters may consist of
several floors, and  each floor begins with an ascending run
of all its left-to-right maxima.
More precisely, it can be always written in the form of
\[
w=w_1 w_2 \cdots w_{i_1}w_{i_1+1}\cdots  w_{i_{k-1}+1} \cdots w_{i_{k}} w_{i_{k}+1} \cdots w_x b b_k \cdots b_1.
\]
For convenience, we write this type of block decomposition of $w$ as type I.
See Fig.~\ref{fig:312214} for a transparent illustration of this decomposition. Note that $\Br(w)=\{w_{i_1},\ldots,w_{i_k},w_x\}$ and $\br(w)=k+1$.

Given a $\{3142,3241\}$-avoiding word of different letters
$v=v_1v_2 \cdots v_n$,   assume that $j_1=1$ and
\[\Vlrmax(v)=\{v_{j_1},v_{j_1+1},\ldots, v_{j_1+l_1},
\ldots, v_{j_k},v_{j_k+1},\ldots,v_{j_k+l_k},
 v_{j_{k+1}},\ldots, v_{j_{k+1}+l_{k+1}}
\}.\]
We write
\[v=v_{j_1}v_{j_1+1}\cdots v_{j_1+l_1} d_1
\cdots v_{j_k} v_{j_k+1}\cdots v_{j_k+l_k}d_k
 v_{j_{k+1}}\cdots v_{j_{k+1}+l_{k+1}}d_{k+1},\]
where $d_1, d_2, \ldots, d_k, d_{k+1}$  are blocks of consecutive letters of $v$. For convenience, we call this type of  block decomposition  as type II.
Then, we have the following proposition (see Fig.~\ref{fig:312421} for the visualization).

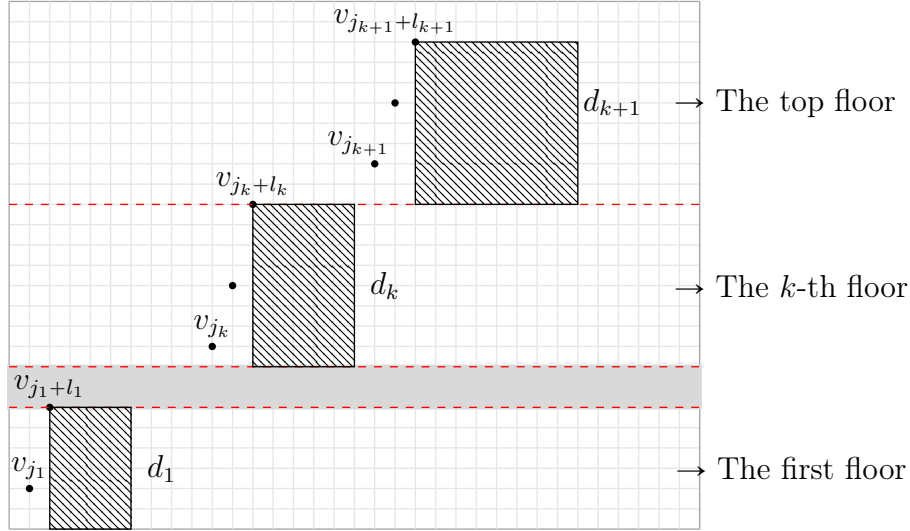
\begin{figure}
\begin{center}
\begin{tikzpicture}[line width=0.5pt,scale=0.27]
\coordinate (O) at (0,0);
\coordinate (O1) at (0,-1);
\coordinate (C) at (1,-25);
\draw [gray!70] (O)--++(34,0);
\draw [gray!20] (O)++(0,-1)--++(34,0);
\draw [gray!20] (O)++(0,-2)--++(34,0);
\draw [gray!20] (O)++(0,-3)--++(34,0);
\draw [gray!20] (O)++(0,-4)--++(34,0);
\draw [gray!20] (O)++(0,-5)--++(34,0);
\draw [gray!20] (O)++(0,-6)--++(34,0);
\draw [gray!20] (O)++(0,-7)--++(34,0);
\draw [gray!20] (O)++(0,-8)--++(34,0);
\draw [gray!20] (O)++(0,-9)--++(34,0);
\draw [gray!20] (O)++(0,-10)--++(34,0);
\draw [gray!20] (O)++(0,-11)--++(34,0);
\draw [gray!20] (O)++(0,-12)--++(34,0);
\draw [gray!20] (O)++(0,-13)--++(34,0);
\draw [gray!20] (O)++(0,-14)--++(34,0);
\draw [gray!20] (O)++(0,-15)--++(34,0);
\draw [gray!20] (O)++(0,-16)--++(34,0);
\draw [gray!20] (O)++(0,-17)--++(34,0);
\draw [gray!20] (O)++(0,-18)--++(34,0);
\draw [gray!20] (O)++(0,-20)--++(34,0);
\draw [gray!20] (O)++(0,-21)--++(34,0);
\draw [gray!20] (O)++(0,-22)--++(34,0);
\draw [gray!20] (O)++(0,-23)--++(34,0);
\draw [gray!20] (O)++(0,-24)--++(34,0);
\draw [gray!20] (O)++(0,-25)--++(34,0);
\draw [gray!70] (O)++(0,-26)--++(34,0);
\draw [gray!70] (O)--++(0,-26);
\draw [gray!20] (O)++(1,0)--++(0,-26);
\draw [gray!20] (O)++(2,0)--++(0,-26);
\draw [gray!20] (O)++(3,0)--++(0,-26);
\draw [gray!20] (O)++(4,0)--++(0,-26);
\draw [gray!20] (O)++(5,0)--++(0,-26);
\draw [gray!20] (O)++(6,0)--++(0,-26);
\draw [gray!20] (O)++(7,0)--++(0,-26);
\draw [gray!20] (O)++(8,0)--++(0,-26);
\draw [gray!20] (O)++(9,0)--++(0,-26);
\draw [gray!20] (O)++(10,0)--++(0,-26);
\draw [gray!20] (O)++(11,0)--++(0,-26);
\draw [gray!20] (O)++(12,0)--++(0,-26);
\draw [gray!20] (O)++(13,0)--++(0,-26);
\draw [gray!20] (O)++(14,0)--++(0,-26);
\draw [gray!20] (O)++(15,0)--++(0,-26);
\draw [gray!20] (O)++(16,0)--++(0,-26);
\draw [gray!20] (O)++(17,0)--++(0,-26);
\draw [gray!20] (O)++(18,0)--++(0,-26);
\draw [gray!20] (O)++(19,0)--++(0,-26);
\draw [gray!20] (O)++(20,0)--++(0,-26);
\draw [gray!20] (O)++(21,0)--++(0,-26);
\draw [gray!20] (O)++(22,0)--++(0,-26);
\draw [gray!20] (O)++(23,0)--++(0,-26);
\draw [gray!20] (O)++(24,0)--++(0,-26);
\draw [gray!20] (O)++(25,0)--++(0,-26);
\draw [gray!20] (O)++(26,0)--++(0,-26);
\draw [gray!20] (O)++(27,0)--++(0,-26);
\draw [gray!20] (O)++(28,0)--++(0,-26);
\draw [gray!20] (O)++(29,0)--++(0,-26);
\draw [gray!20] (O)++(30,0)--++(0,-26);
\draw [gray!20] (O)++(31,0)--++(0,-26);
\draw [gray!20] (O)++(32,0)--++(0,-26);
\draw [gray!20] (O)++(33,0)--++(0,-26);
\draw [gray!20] (O)++(34,0)--++(0,-26);
\draw [gray!70] (O)++(34,0)--++(0,-26);

\draw [fill=gray!30,ultra thick,gray!30] (0,-18) rectangle (34,-20);

\fill[black!100] (C)++(0,1) circle(1ex)++(1,4) circle(1ex)
++(8,3) circle(1ex)++(1,3) circle(1ex)
++(1,4) circle(1ex)
++(6,2) circle(1ex)
++(1,3) circle(1ex)
++(1,3) circle(1ex);

\draw [dashed,red] (C)++(-1,5)--++(34,0);
\draw [dashed,red] (C)++(-1,7)--++(34,0);
\draw [dashed,red] (C)++(-1,15)--++(34,0);

\draw[pattern=north west lines] (2,-20) rectangle (6,-26);
\draw[pattern=north west lines] (12,-10) rectangle (17,-18);
\draw[pattern=north west lines] (20,-2) rectangle (28,-10);
\path (7.5,-23)  node {$d_1$}
++(11,9)  node {$d_k$}
++(11.2,9)  node {$d_{k+1}$};

\path (2,-19)  node {$v_{j_1+l_1}$}
++(-1,-4) node {$v_{j_1}$}
++(1,4)
++(8,3) node {$v_{j_k}$}
++(2,7) node {$v_{j_k+l_k}$};

\path (19,-1)  node {$v_{j_{k+1}+l_{k+1}}$}
++(-1.8,-6) node {$v_{j_{k+1}}$};

\path (38.5,-23)  node {$\boldsymbol{\rightarrow}$ The first floor};
\path (38.5,-14)  node {$\boldsymbol{\rightarrow}$ The $k$-th floor};
\path (38.2,-5)  node {$\boldsymbol{\rightarrow}$ The top floor};
\end{tikzpicture}
\caption{The structure of a $\{3142, 3241\}$-avoiding word of different letters.}\label{fig:312421}
\end{center}
\end{figure}

\begin{proposition}\label{prop:stru2}
Suppose that $v=v_1v_2 \cdots v_n \in \W_n({3142,3241})$ with block decomposition of type II, then
all the elements of $d_1$ are smaller than $v_{j_1+l_1}$ and
all the elements of $d_s$ are smaller than $v_{j_s+l_s}$ and larger than $v_{j_{s-1}+l_{s-1}}$ for $2 \leq s \leq k+1$.
\end{proposition}

\begin{proof}
Suppose that $x$ is an element of $d_s$ with  $2 \leq s \leq k+1$,
clearly we have  $x<v_{j_s+l_s}$.  Assume to the contrary that
$x<v_{j_{s-1}+l_{s-1}}$, then $v_{j_{s-1}+l_{s-1}} v_{j_{s-1}+l_{s-1}+1} v_{j_s+l_s} x$ will form a pattern of $3142$ or $3241$, a contradiction.
Thus, we have $x>v_{j_s+l_s}$.
Notice that $v_{j_1+l_1}$ and  $v_{j_2}$ are the left-to-right maxima of $v$, elements of $d_1$ are certainly smaller than $v_{j_1+l_1}$. This completes the proof.
\end{proof}

Based on Proposition \ref{prop:stru2}, we see that a
$\{3142,3241\}$-avoiding word of different letters also consists of
several floors, and  each floor begins with an ascending run
of all its left-to-right maxima.
See Fig.~\ref{fig:312421} for a transparent illustration of this decomposition.

Now, we are ready to give the description of $\varphi$.
Given a $\{3124,3214\}$-avoiding word of different letters $w=w_1 w_2 \cdots w_n$ with
block decomposition of type I, recall that $w_x=\max(w)$. We construct $\varphi(w)$ through the following three cases:
\begin{itemize}
  \item If $n=0$, then define $\varphi(\emptyset)=\emptyset$.

  \item If $w$ begins with an ascending run, the end of which is $\max(w)$, then define $\varphi(w)$ to be  the word obtained by inserting $\max(w)$ into $\varphi(w_1\cdots w_{x-1}w_{x+1}\cdots w_n)$ at position $x$.

  \item If it is not the cases above, then define $\varphi(w)$ to be the word
  \[
  \varphi(w_1 \cdots w_{i_1} w_{i_1+1} b_1)   \cdots
  \varphi(w_{i_{k-1}+2} \cdots w_{i_k} w_{i_k+1} b_k)
  \varphi(w_{i_{k}+2} \cdots w_{x} b)
  \]
\end{itemize}

To show that $\varphi$ is well defined, we need to verify that $\varphi(w)$ avoids $\{3142,3241\}$. This can be easily seen by induction on the length of the word in view of
item 2 in Proposition \ref{prop:stru1}.
To prove that $\varphi$ is a bijection, we give its inverse $\psi$.
Given $v=v_1 v_2 \cdots v_n \in \W_n(3142,3241)$,
with block decomposition of type II, notice that
 $\max(v)=v_{j_{k+1}+l_{k+1}}$.
We construct $\psi(v)$ through the following three cases:
\begin{itemize}
  \item If $n=0$, then define $\psi(\emptyset)=\emptyset$.

  \item If $v$ begins with an ascending run, the end of which is $\max(v)$, then define $\psi(v)$ to be  the word obtained by inserting $\max(v)$ into $\psi(v_1\cdots v_{j_{k+1}+l_{k+1}-1}v_{j_{k+1}+l_{k+1}+1}\cdots v_n)$ at position $j_{k+1}+l_{k+1}$.

  \item If it is not the cases above, then let
   $\psi(v_{j_s} \cdots v_{j_s+l_s}d_s) = F_sL_s$ with $1 \leq s \leq k+1$, where $F_s$ consists of the first $l_s+2$ elements of $\psi(v_{j_s} \cdots v_{j_s+l_s}d_s)$ and $L_s$
   consists of the remaining ones. Define $\psi(v)$  to be the word $F_1 F_2 \cdots F_{k+1} L_{k+1} L_{k} \cdots L_1$.
\end{itemize}

\begin{example}
Assume that $\pi \in \SS_{20}(3124,3214)$ and
\[\pi=2 \, 6\, 4\, 7\, 10\, 14\, 9\, 15\, 17\, 20\, 19\, 16\, 18\, 11\, 12\, 13\, 8\, 3\, 5\, 1.\]
In the notation of block decomposition of type I, we have
$k=2,\,i_1=2,\,i_2=6$, $x=10$, $b_1=3\,5\,1$, $b_2=11\,12\,13\,8$
and
$b=19\,16\,18$.
By the construction of $\varphi$, we may deduce that
\begin{align*}
  \varphi(\pi)&=\varphi(2 \, 6\, 4\, 3\, 5\, 1)
\varphi(7 \, 10 \, 14 \, 9 \, 11\, 12\, 13\,8)
\varphi(15\, 17\, 20\, 19\, 16\, 18)\\
  &= 2\, 6\, 4\, 3\, 1\, 5\, 7\, 10\, 14\, 9\, 8\, 11\, 12\, 13\, 15\, 17\, 20\, 19\, 16\, 18.
  \end{align*}

On the other hand, assume that $p \in \SS_{20}(3142,3241)$ and
\[
p=2\, 6\, 4\, 3\, 1\, 5\, 7\, 10\, 14\, 9\, 8\, 11\, 12\, 13\, 15\, 17\, 20\, 19\, 16\, 18.
\]

In the notation of block decomposition of type II, we have
$k=2$, $j_1=1$, $l_1=1$, $j_2=7$, $l_2=2$, $j_3=15$,
$l_3=2$, $d_1=4 \, 3\, 1\,5$, $d_2=9\, 8\, 11\, 12\, 13$,
$d_3=19 \, 16\, 18$. By the construction of $\psi$, we
have $\psi(2\,6\,4\,3\,1\,5)=2 \, 6\, 4\, 3\, 5\, 1$,
$\psi(7\,10\,14\,9\,8\,11\,12\,13)=7 \, 10 \, 14 \, 9 \, 11\, 12\, 13\,8$ and $\psi(15\,17\,20\,19\,16\,18)=15\, 17\, 20\, 19\, 16\, 18$. It follows that $F_1=2 \, 6\, 4$, $L_1=3\, 5\, 1$,
$F_2=7 \, 10 \, 14 \, 9$, $L_2=11\, 12\, 13\,8$,
$F_3=15\, 17\, 20\, 19$, $L_3=16\, 18$.
Hence, we deduce that
\begin{align*}
  \psi(p)&=F_1 F_2 F_3 L_3 L_2 L_1\\
  &=2 \, 6\, 4\, 7\, 10\, 14\, 9\, 15\, 17\, 20\, 19\, 16\, 18\, 11\, 12\, 13\, 8\, 3\, 5\, 1.
  \end{align*}

\end{example}

The following proposition of $\varphi$ makes sure that $\psi$ is the inverse of $\varphi$, which may be easily checked by induction.

\begin{proposition}\label{prop:vardes}
 Let $w$ be a $\{3124,3214\}$-avoiding word of different letters  with block decomposition of type I.
If $w_{i} w_{i+1}$ is a descent of $w$ for $1 \leq i \leq x$, then $w_{i}$ and $w_{i+1}$  remain  adjacent  in $\varphi(w)$.
\end{proposition}

Now, to complete the proof of
Theorem \ref{thm:fourplus}, we need to justify the following Proposition.

\begin{proposition}\label{prop:var}
 Given a $\{3124,3214\}$-avoiding word of different letters $w=w_1 w_2 \cdots w_n$ with block decomposition of type I, denote $v=\varphi(w)$ and we have
 \begin{itemize}
   \item [1.] $\Iar(w)=\Iar(v)$.
   \item [2.] $\Vlrmax(w)=\Vlrmax(v)$.
   \item [3.] $\Ides(w)=\Ides(v)$.
  \item [4.] $\Vlrmin(w)=\Vlrmin(v)$.
  \item [5.] $\Br(w)=\Br(v)$.
   \end{itemize}
\end{proposition}
\begin{proof}
We use induction on $n$. Clearly,  each item holds for $w=\emptyset$. Now suppose that this proposition holds for all $\{3124,3214\}$-avoiding words of different letters with length $n-1$. To justify it for $n$, we consider two cases.
\begin{itemize}
  \item  If
   $w$ begins with an ascending run ending with  $\max(w)$,
    recall that $\varphi(w)$ is obtained by inserting $w_x=\max(w)$ into $\varphi(w_1\cdots w_{x-1}w_{x+1}\cdots w_n)$ at position $x$.
     By  the induction hypothesis, $w_1w_2\cdots w_{x-1}$ with $w_1< w_2 < \cdots< w_{x-1}$ is just the initial sequence of $\varphi(w_1\cdots w_{x-1}w_{x+1}\cdots w_n)$, and hence $w_1w_2\cdots w_{x-1}w_x$
     is the initial run of $v$. Thus, items $1$ and $2$ are verified.
     Notice that the relative positions of $\max(w)$ and the second largest element of $w$ will not be changed by the map $\varphi$, then
     item $3$ follows directly from the induction hypothesis.
    By checking two cases, $x=1$ or not, we may easily deduce that $\Vlrmin(v)=\{\max(w)\} \cup \Vlrmin(\varphi(w_2 \cdots w_n))$
    or $\Vlrmin(v)=\Vlrmin(\varphi(w_1 \cdots w_{x-1} w_{x+1} \cdots w_n))$, respectively. Based on the induction, item $4$ is confirmed. As $\Iar(w)=\Iar(v)$, we have $\Br(w)=\Br(v)=\{w_x\}$ and so item $5$ holds.

  \item  If it is not the case above,  recall that
  \[v=\varphi(w_1 \cdots w_{i_1} w_{i_1+1} b_1)   \cdots
  \varphi(w_{i_{k-1}+2} \cdots w_{i_k} w_{i_k+1} b_k)
  \varphi(w_{i_{k}+2} \cdots w_{x} b).\]
  Following from the fact that the initial ascending run and the left-to-right minima of $v$ are just those of $\varphi(w_1 \cdots w_{i_1} w_{i_1+1} b_1)$, we obtain items $1$ and $4$. Moreover,
  it is not hard to see that the left-to-right maxima of $v$ are the union of the elements in the initial ascending runs of $\varphi(w_1 \cdots w_{i_1} w_{i_1+1} b_1),  \ldots,
  \varphi(w_{i_{k-1}+2}  \cdots  w_{i_k} w_{i_k+1} b_k)$ and
  $\varphi(w_{i_{k}+2}, \cdots w_{x} b)$.
  By the induction hypothesis, we deduce that
  \[\Vlrmax(v)=\{w_1, \ldots, w_{i_1},w_{i_1+2}, \ldots, w_{i_2},\ldots, w_{i_{k}+2}, \ldots, w_{x} \},\]
  which equals to $\Vlrmax(w)$. Hence, item $2$ follows.
  As for item $3$, firstly, we note that the change of the relative order of the blocks $b, b_1, \ldots, b_k$ under $\varphi$ will bring no difference to the relative order of numerically adjacent elements of $w$.
  Secondly, by induction we see that
  the relative order of numerically adjacent elements in each floor of $w$ remain the same. Combining these two facts, item $3$ follows. Finally, it is plain to see that $\Br(w)=\{w_{i_1},\ldots,w_{i_k},w_x\}=\Br(v)$ and item $5$ follows.
\end{itemize}
The proof is now completed.
\end{proof}

\subsection{Type II block decomposition and the generating function} This section aims to compute the generating function
$$
\sum_{n\geq1}z^n\sum_{\pi\in\SS_n(3142,3241)}t^{\ides(\pi)}p^{\lrmax(\pi)}q^{\lrmin(\pi)}
$$
for $(3142,3241)$-avoiding permutations using the type II block decomposition   in Proposition~\ref{prop:stru2}.
Recall that $\iar(\pi)$ is the length of the initial ascending run of $\pi$.   Let us introduce
$$
\B_n:=\{\pi\in\SS_n(3142,3241): \pi\neq\id_n, \iar(\pi)=\lrmax(\pi)\},
$$
where $\id_n=12\cdots n$ is the identity permutation of length $n$.
Define
\begin{align*}
S=S(x,t,p,q;z)&=\sum_{n\geq1}z^n\sum_{\pi\in\SS_n(3142,3241)}x^{\iar(\pi)}t^{\ides(\pi)}p^{\lrmax(\pi)}q^{\lrmin(\pi)}=\sum_{k\geq1} S_k(t,p,q;z)x^k,\\
B=B(x,t,p,q;z)&=\sum_{n\geq2}z^n\sum_{\pi\in\B_n}x^{\iar(\pi)}t^{\ides(\pi)}p^{\lrmax(\pi)}q^{\lrmin(\pi)}=\sum_{k\geq1} B_k(t,p,q;z)x^k,\\
I=I(x,p,q;z)&=\sum_{n\geq1}z^nx^{\iar(\id_n)}p^{\lrmax(\id_n)}q^{\lrmin(\id_n)}=\frac{xpqz}{1-xpz}.
\end{align*}
For convenience, set $S_k(1)=S_k(t,1,q;z)$ and $B_k(1)=B_k(t,1,q;z)$.
By  the type II block decomposition of $(3142,3241)$-avoiding permutations (see Fig.~\ref{fig:312421}) we have
\begin{equation}\label{eq:S}
S=I+B(S(1,t,p,1)+1).
\end{equation}
On the other hand, any permutation in $\B_n$ can be obtained in one of the following cases:
\begin{enumerate}
\item inserting $n$ into a position that is not the rightmost one in $\id_{n-1}$;
\item inserting $n$ at the beginning or after one of the letters in the initial ascending run of a permutation in $\B_{n-1}$;
\item inserting $n$ at the beginning or after one of the letters in the initial ascending  run of a permutation in $\SS_{n-1}(3142,3241)\setminus(\B_{n-1}\cup\{\id_{n-1}\})$.
\end{enumerate}
It then follows that
\begin{align*}
B&=z\sum_{k\geq1}B_k(1)(tpqx+tp^2x^2+tp^3x^3+\cdots+tp^kx^k+p^{k+1}x^{k+1})\\
&\quad+z\sum_{k\geq1}(S_k(1)-B_k(1))(tpqx+tp^2x^2+tp^3x^3+\cdots+tp^kx^k+tp^{k+1}x^{k+1})\\
&\quad-\sum_{k\geq1}tp^{k+1}x^{k+1}z^{k+1},
\end{align*}
which is simplified to
\begin{equation}\label{eq:B}
B=\biggl(\frac{tpxz}{1-px}+tpxz(q-1)\biggr)S(1,t,1,q)-\frac{tp^2x^2z}{1-px}S(px,t,1,q)+(1-t)pxzB-\frac{tp^2x^2z^2}{1-pxz}.
\end{equation}
Combining~\eqref{eq:S} and~\eqref{eq:B} results in
\begin{theorem} The generating function $S$ satisfies the algebraic equation
\begin{equation*}\label{eq:3142-3241}
\frac{(S-I)(1-(1-t)pxz)}{1+S(1,t,p,1)}=\biggl(\frac{tpxz}{1-px}+tpxz(q-1)\biggr)S(1,t,1,q)-\frac{tp^2x^2z}{1-px}S(px,t,1,q)-\frac{tp^2x^2z^2}{1-pxz}.
\end{equation*}
\end{theorem}

\section{The bijection $\alpha$ and proof of Theorem~\ref{thm:fiveplus}}
\label{sec:3}
Based on the bijection $\varphi$ between $\SS_n(3124,3214)$ and $\SS_n(3142,3241)$, this section is devoted to the
recursive construction of  our main bijection
$$
\alpha: \W_n(31245,32145,31254,32154)\rightarrow\W_n(31425,32415,31524,32514),
$$
which preserves  the quintuple of set-valued statistics $(\Ides,\Vlrmax,\Vlrmin,\Vrlmax,\Iar)$. This confirms Theorem~\ref{thm:fiveplus} in the sense of isomorphism.

{\bf The construction of $\alpha$.} Given  $w=w_1 w_2 \cdots w_n \in \W_n(31245,32145,31254,32154)$,
we always write $\Vlrmax(w)=\{w_{l_1}, w_{l_2},\ldots, w_{l_s} \}$ and
 $\Vrlmax(w)=\{w_{r_1}, w_{r_2},\ldots, w_{r_t} \}$ with $1=l_1<\cdots<l_s=r_1< \cdots <r_t=n$.
 Then $\alpha(w)$ can be constructed according to the following cases:
\begin{itemize}
  \item If $n=0$, then define $\alpha(\emptyset)=\emptyset$.

  \item If $\lrmax(w)=1$ and $\rlmax(w) \geq 1$, then
        $\alpha(w)=\max(w) \alpha(w_2 \cdots w_n)$.

  \item If $\lrmax(w)>1$ and $\rlmax(w) = 1$, then
        $\alpha(w)= \varphi(w_1 \cdots w_{n-1})\max(w)$, where $\varphi$ is the bijection introduced in Theorem~\ref{thm:fourplus}.

  \item If $\lrmax(w)>1$ and $\rlmax(w)>1$, then we consider two cases.
      \begin{itemize}
        \item[(a)] If $l_s=s$,
         then define $\alpha(w)$ to be  the word obtained by inserting $\max(w)$, namely $w_{l_s}$, into $\alpha(w_1 \cdots w_{l_s-1} w_{l_s+1} \cdots w_n)$ at position $l_s$.
        \item[(b)] If $l_s>s$,
        we further consider three cases.
\begin{itemize}
          \item[(b1)] If $w_{l_{s-1}}<w_{r_2}$, then let $u=\alpha(w_1 \cdots w_{l_s-1}w_{l_s+1} \cdots w_n).$
           When $l_{s-1}+1=l_s$,
           we construct $\alpha(w)$ by inserting $w_{l_{s}}$ into $u$ just after 
           the $(s-1)$-th left-to-right maximum of $u$.
           When $l_{s-1}+1<l_s$,
           we construct $\alpha(w)$ by inserting $w_{l_{s}}$ into $u$ just before the $s$-th left-to-right maximum of $u$.

           \item[(b2)] If $w_{l_{s-1}}>w_{r_2}$ and $s>2$,
           then denote  $u=\alpha(w_1 \cdots w_{l_{s-1}-1}w_{l_{s-1}+1} \cdots w_n)$.
           When $l_{s-2}+1=l_{s-1}$,
           we construct $\alpha(w)$ by inserting $w_{l_{s-1}}$ into $u$ just after the $(s-2)$-th left-to-right maximum of $u$.
           When $l_{s-2}+1<l_{s-1}$,
            insert $w_{l_{s-1}}$ into $u$ just before the $(s-1)$-th left-to-right maximum of $u$ and we obtain  $\alpha(w)$.

          \item[(b3)] If $w_{l_{s-1}}>w_{r_2}$ and $s=2$,
           then $\alpha(w)$ can be obtained by inserting
           $w_1$ at the beginning of $\alpha(w_2 \cdots w_n)$.
           \end{itemize}
              \end{itemize}
 \end{itemize}

To prove that $\alpha$ is a well-defined bijection, we need to analyze
the structure of the words in $\W_n(31245,32145,$ $31254,32154)$. Given such a word $w$, we focus on the cases when
 $\lrmax(w)>1$, $\rlmax(w)>1$ and $l_s>s$ (i.e., case~(b) above),
 which we refer to as the non-trivial cases.

\begin{lemma}\label{prop:<}
Assume that $w$ is a non-trivial $\{31245,32145,31254,32154\}$-avoiding word of different letters with $w_{l_{s-1}}<w_{r_2}$ and $l_s>s$,  there are totally three types:
\begin{itemize}
  \item[I-1.]  $l_s=l_{s-1}+1$.
  \item[I-2.]  $l_s=l_{s-1}+2$ and $w_j<w_{l_{s-1}}$ with $l_s<j<r_2$.
  \item[I-3.]  $l_s=l_{s-1}+2$ and there is an integer $k$ between $l_s$ and $r_2$ such that
      $w_j>w_{l_{s-1}}$ with $l_s<j\leq k$ and $w_j<w_{l_{s-1}}$ with $k<j< r_2$.

\end{itemize}
\end{lemma}

\begin{proof}
Firstly, we show that there is at most one element between $w_{l_{s-1}}$
and $w_{l_s}$ in $w$. Otherwise,
we have  $l_s>l_{s-1}+2$. Then,
$w_{l_{s-1}}w_{l_{s-1}+1}w_{l_{s-1}+2}w_{l_{s}}w_{r_2}$ will form a
  $31254$ or $32154$-pattern, a contradiction.
 Hence, we deduce that $l_s=l_{s-1}+1$ or $l_s=l_{s-1}+2$.

  When $l_s=l_{s-1}+2$, one possibility is that there is no element  between $w_{l_s}$
  and $w_{r_2}$ larger than $w_{l_{s-1}}$. If not, assume that
  $k=\max\{j \colon w_j>w_{l_{s-1}} \text{~and~} l_{s-1}<j<r_2\}$.
  We claim that $w_j>w_{l_{s-1}}$ for $l_{s-1}<j\leq k$.
  Otherwise, assume to the contrary that  there exists an integer $o$ such that $l_{s-1}<o<k$ and
  $w_o<w_{l_{s-1}}$, then $w_{l_{s-1}}w_{l_{s-1}+1}w_o w_k w_{r_2}$
  will give an instance of $31245$ or $32145$. This contradicts with the fact that $w \in \W_n(31245,32145,31254,32154)$. The claim is verified and we complete the proof.
\end{proof}

Based on Lemma~\ref{prop:<}, we give the corresponding graphical descriptions in Fig.~\ref{fig:structure-w-I}.

\begin{figure}[!htbp]
\begin{center}
\begin{tikzpicture}[line width=0.5pt,scale=0.27]
\coordinate (O) at (0,0);
\coordinate (O1) at (-18,0);
\coordinate (O2) at (-37,0);
\fill[black!100] (O)circle(1ex)++(1,-4) circle(1ex)++(1,9.5) circle(1ex)++(1,-3) circle(1ex)++(5,2) circle(1ex);
\draw[pattern=north west lines] (3,0) rectangle (5,4.5);
\draw[pattern=north west lines] (5,-7) rectangle (8,0);

\path (-0.5,0.9)  node {\tiny{$w_{l_{s-1}}$}}
++(1,-4) node {\tiny{$w_{l_{s-1}+1}$}}
++(1.4,9.5) node {\tiny{$w_{l_{s}}=w_{l_{s-1}+2}$}}
++(6.5,-1) node {\tiny{$w_{r_2}$}}
++(-4,-14.5) node {\small{Type I-3}};

\fill[black!100] (O1)circle(1ex)++(1,-4) circle(1ex)++(1,9.5) circle(1ex)++(1,-3) ++(5,2) circle(1ex);
\draw[pattern=north west lines] (-16,-7) rectangle (-10,0);

\path (-18.5,0.9)  node {\tiny{$w_{l_{s-1}}$}}
++(0.4,-4) node {\tiny{$w_{l_{s-1}+1}$}}
++(2,9.5) node {\tiny{$w_{l_{s}}=w_{l_{s-1}+2}$}}
++(6.5,-1) node {\tiny{$w_{r_2}$}}
++(-4,-14.5) node {\small{Type I-2}};

\fill[black!100] (O2)++(1,0)circle(1ex)++(1,5.5) circle(1ex)++(1,-3) ++(5,2) circle(1ex);

\draw[pattern=north west lines] (-35,-7) rectangle (-29,4.5);

\path (-36.5,0.9)  node {\tiny{$w_{l_{s-1}}$}}
++(0.4,-4)
++(2,9.5) node {\tiny{$w_{l_{s}}=w_{l_{s-1}+1}$}}
++(5.5,-1) node {\tiny{$w_{r_2}$}}
++(-3,-14.5) node {\small{Type I-1}};

\end{tikzpicture}
\caption{The structure of a non-trivial $(31245,32145,31254,32154)$-avoiding word with $w_{l_{s-1}}<w_{r_2}$ .}
\label{fig:structure-w-I}
\end{center}
\end{figure}
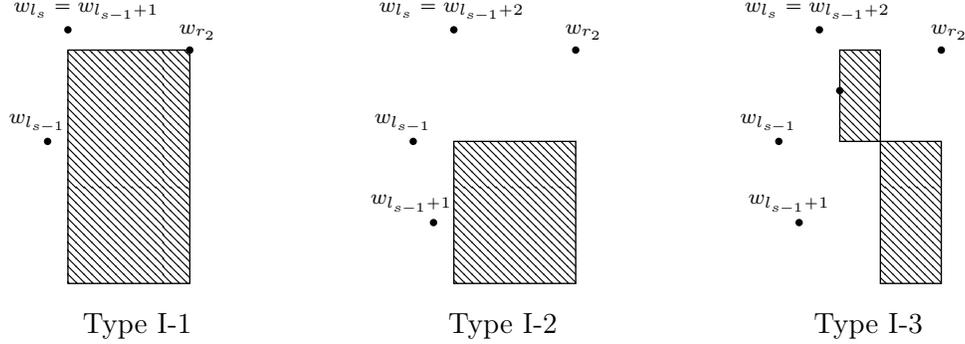

Similar derivation may lead us to the following proposition.
 Its corresponding graphical descriptions are presented in Fig.~\ref{fig:structure-w-II}, where the gray boxes represent
consecutive increasing sequences which might be empty.

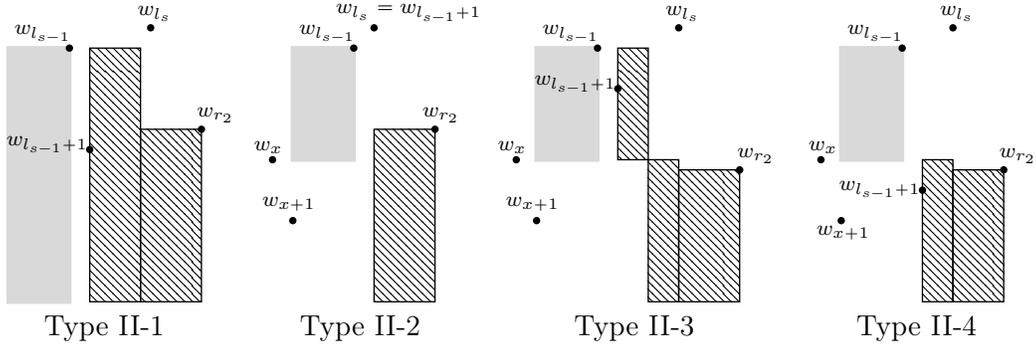
\begin{figure}[!htbp]
\begin{center}
\begin{tikzpicture}[line width=0.5pt,scale=0.27]
\coordinate (O) at (0,0);
\coordinate (O1) at (14,0);
\coordinate (O2) at (26,0);
\coordinate (O3) at (41,0);
\draw[pattern=north west lines] (7.5,-7) rectangle (10.5,1.5);
\draw[pattern=north west lines] (5,-7) rectangle (7.5,5.5);
\draw [fill=gray!30,ultra thick,gray!30] (1,-7) rectangle (4,5.5);

\fill[black!100] (O)++(4,5.5) circle(1ex)++(1,-5) circle(1ex)
++(3,6) circle(1ex)++(2.5,-5) circle(1ex);

\path (2.7,6.2)  node {\tiny{$w_{l_{s-1}}$}}
++(0.2,-5.5)node {\tiny{$w_{l_{s-1}+1}$}}
++(5.3,6.5) node {\tiny{$w_{l_s}$}}
++(3,-5) node {\tiny{$w_{r_2}$}}
++(-5.5,-10.5) node {\small{Type II-1}};

\draw[pattern=north west lines] (19,-7) rectangle (22,1.5);
\draw [fill=gray!30,ultra thick,gray!30] (15,0) rectangle (18,5.5);

\fill[black!100] (O1)circle(1ex)++(1,-3) circle(1ex)++(3,8.5) circle(1ex)++(1,1) circle(1ex)++(3,-5) circle(1ex);

\path (16.7,6.2)  node {\tiny{$w_{l_{s-1}}$}}
++(4,1) node {\tiny{$w_{l_s}=w_{l_{s-1}+1}$}}
++(1.6,-5) node {\tiny{$w_{r_2}$}}
++(-4,-10.5) node {\small{Type II-2}};

\path (13.8,0.6)  node {\tiny{$w_{x}$}}
++(1,-2.8) node {\tiny{$w_{x+1}$}};

\draw[pattern=north west lines] (31,0) rectangle (32.5,5.5);
\draw[pattern=north west lines] (32.5,-7) rectangle (34,0);
\draw[pattern=north west lines] (34,-7) rectangle (37,-0.5);
\draw [fill=gray!30,ultra thick,gray!30] (27,0) rectangle (30,5.5);

\fill[black!100] (O2)circle(1ex)++(1,-3) circle(1ex)++(3,8.5) circle(1ex)++(1,-2) circle(1ex)++(3,3) circle(1ex)++(3,-7) circle(1ex);

\path (28.7,6.2)  node {\tiny{$w_{l_{s-1}}$}}
++(0.2,-2.5)node {\tiny{$w_{l_{s-1}+1}$}}
++(5.3,3.5) node {\tiny{$w_{l_s}$}}
++(3.6,-7) node {\tiny{$w_{r_2}$}}
++(-6,-8.5) node {\small{Type II-3}};

\path (25.8,0.6)  node {\tiny{$w_{x}$}}
++(1,-2.8) node {\tiny{$w_{x+1}$}};

\draw [fill=gray!30,ultra thick,gray!30] (42,0) rectangle (45,5.5);
\draw[pattern=north west lines] (46,-7) rectangle (47.5,0);
\draw[pattern=north west lines] (47.5,-7) rectangle (50,-0.5);

\fill[black!100] (O3)circle(1ex)++(1,-3) circle(1ex)++(3,8.5) circle(1ex)++(1,-7)circle(1ex)
++(1.5,8) circle(1ex)++(2.5,-7) circle(1ex);

\path (43.7,6.2)  node {\tiny{$w_{l_{s-1}}$}}
++(0.2,-7.7) node {\tiny{$w_{l_{s-1}+1}$}}
++(3.8,8.7) node {\tiny{$w_{l_s}$}}
++(3,-7) node {\tiny{$w_{r_2}$}}
++(-5,-8.5) node {\small{Type II-4}};

\path (41.1,0.6)  node {\tiny{$w_{x}$}}
++(1,-4.2) node {\tiny{$w_{x+1}$}};
\end{tikzpicture}
\caption{The structure of a non-trivial $(31245,32145,31254,32154)$-avoiding word with $w_{l_{s-1}}>w_{r_2}$ and $s>2$ .}
\label{fig:structure-w-II}
\end{center}
\end{figure}

\begin{lemma}\label{prop:>}
Assume that $w$ is a non-trivial $\{31245,32145,31254,32154\}$-avoiding word of different letters with $w_{l_{s-1}}>w_{r_2}$, $l_s>s$ and $s>2$, let $x=\max(\{j \colon w_j>w_{j+1} ~\text{and}~j<l_{s-1} \}\cup\{0\})$.
Then there are four types:
\begin{itemize}
  \item[II-1.] $x=0$.
  \item[II-2.] $x \neq 0$ and $l_s=l_{s-1}+1$.
  \item[II-3.] $x \neq 0$, $l_s>l_{s-1}+1$ and there is an integer $k$,
  $$
  k=\max\{i: l_{s-1}<i<l_s, w_x<w_i<w_{s-1}\},
  $$
       such that $w_{l_{s-1}}>w_j>w_x$ for $l_{s-1}<j\leq k$
       and $w_j<w_x$ for $k<j< l_s$. (If $k<l_s-1$, then $w_x>w_{r_2}$.)
  \item[II-4.] $x \neq 0$, $l_s>l_{s-1}+1$ and $w_j<w_x$ with $l_{s-1}<j<l_s$. (This implies that $w_x>w_{r_2}$.)
\end{itemize}
\end{lemma}

To show that $\alpha$ is  well defined, we need the proposition below.

 \begin{proposition}\label{prop:alpha}
 Given a $(31245,32145,31254,32154)$-avoiding word of different letters $w$, denote $v=\alpha(w)$ and we have
 \begin{itemize}
   \item [1.] $\Vlrmax(w)=\Vlrmax(v)$.
   \item [2.] $w$ and $v$ have the same initial ascending run.
   \item [3.] $v$ avoids patterns in $\{31425,32415,31524,32514\}$.
   \end{itemize}
Consequently,  the map $\alpha$ is  well defined.
 \end{proposition}

 \begin{proof}
 We  proceed  by induction on the length of the word. When $w=\emptyset$,
 it certainly holds. Assume that this proposition  holds for
 words in $\W_{n-1}(31245,32145,31254,32154)$, we aim to verify that it holds for those of length $n$. Assume that $w=w_1 w_2 \cdots w_{n}$, we consider the following cases.

 If $\lrmax(w)=1$ and $\rlmax(w) \geq 1$, then $w_1=\max(w)$ and
 $v=w_1 \alpha(w_2 \cdots w_n)$. It is easy to check that  $\Vlrmax(w)=\Vlrmax(v)=\{w_1\}$ and $\Iar(w)=\Iar(v)$. By the induction hypothesis, $\alpha(w_2 \cdots w_n)$ is
 $(31425,32415,31524,32514)$-avoiding. Since $w_1$ can not play the role of $3$ in any pattern of length five, we deduce that
  $v$ avoids $\{31425,32415,31524,32514\}$.

If $\lrmax(w)>1$ and $\rlmax(w) = 1$, then $w_n=\max(w)$ and
        $\alpha(w)= \varphi(w_1 \cdots w_{n-1})w_n$.
        By Theorem~\ref{thm:fourplus}, we see that the map $\varphi$ keeps the initial ascending run and the statistic $\Vlrmax$.
        Thus, we deduce that items 1 and 2 hold for this case.
         Noticing that
        $w_n$ can only play the role of $5$ in a pattern of length five, but not in a pattern $31425$ or $32415$.
        The fact that
        $v$ avoids $\{31425,32415,31524,32514\}$ following from $\varphi(w_1 \cdots w_{n-1})$ avoids $\{3142,3241\}$ directly.

If $\lrmax(w)>1$, $\rlmax(w)>1$ and $l_s=s$,
it is trivial to check that $\Vlrmax(w)=\Vlrmax(v)=\{w_1,w_2, \ldots, w_{s}\}=\Iar(w)=\Iar(v)$.
Now we need to explain that  $v$ avoids $\{31425,32415,31524,32514\}$.
Based on the facts that
 $\max(v)$ can only play the role of $5$ in a pattern of
length five  and  there is no descents before $\max(w)$ in $v$,
we deduce that any instance of $v$ containing $\max(w)$ avoids
 $\{31425,32415,31524,32514\}$. In view of the fact that $\alpha(w_1 \cdots w_{l_s-1} w_{l_s+1} \cdots w_n)$
 avoids $\{31425,32415,31524,32514\}$
 by the induction hypothesis,  we obtain item 3 for this case.

 If $\lrmax(w)>1$, $\rlmax(w)>1$ and $l_s>s$,
 there are three subcases to be considered.
 \begin{itemize}
 \item  When $w_{l_{s-1}}<w_{r_2}$,  the word $w_1 \cdots w_{l_{s}-1}w_{l_{s}+1}\cdots w_n$ contains  at least $s$ left-to-right maxima,
  say $w_{l_1}, \ldots, w_{l_{s-1}},w_{r_2}$.
  And these elements remain to be the left-to-right maxima
  of $u$. Hence, the $(s-1)$-th and the $s$-th left-to-right
  maximum in the construction of $\alpha$ for this case  do exist.
Based on Lemma~\ref{prop:<}, it is routine to check items $1$ and $2$ through the induction  hypothesis and the construction of $\alpha$ in case~(b1). It  remains to prove item 3. Assume to the contrary that $v$ contains a pattern in $\{31425,32415,31524,32514\}$, we aim to deduce contradictions.
 By the induction hypothesis, we see that $u$ avoids $\{31425,32415,31524,32514\}$.
 So we need only focus on those instances containing the newly inserted element $w_{l_s}$ according to the three structure  types  in Fig.~\ref{fig:structure-w-I}.
 \begin{itemize}
 \item For type I-1, if $v_{g_1}v_{g_2}v_{g_3}v_{g_4}v_{g_5}$ forms a pattern $31425$ or $32415$ of $v$, then clearly $v_{g_4} \neq w_{l_{s-1}}$ and $v_{g_5}=w_{l_s}$.
 It follows that $v_{g_1}v_{g_2}v_{g_3}v_{g_4}w_{l_{s-1}}$ forms
 a pattern $31425$ or $32415$ of $u$, a contradiction.
 If $v_{g_1}v_{g_2}v_{g_3}v_{g_4}v_{g_5}$ forms a pattern $31524$ or $32514$ of $v$, then  we deduce that $v_{g_1}<w_{l_{s-1}}$, $v_{g_2} \neq w_{l_{s-1}}$ and $v_{g_3}=w_{l_s}$. This implies that
 $v_{g_1}v_{g_2}w_{l_{s-1}}v_{g_4}v_{g_5}$ forms a pattern in
 $\{31425,32415,31524,32514\}$ of $u$, a contradiction.
 \item For type I-2 (resp.~I-3), we may similarly check that any instance forming a pattern  in $\{31425,32415,31524,32514\}$ of $v$, which contains $w_{l_s}$, will
 correspond to an instance forming a pattern in $\{31425,32415,31524,32514\}$ of $u$ by changing $w_{l_s}$ into $w_{r_2}$ (resp.~$w_{l_s+1}$).
 This contradicts with the fact that $u$ is $(31425,32415,31524,32514)$-avoiding.
 \end{itemize}
 Thus, we complete the proof of this case.

 \item When  $w_{l_{s-1}}>w_{r_2}$ with $s>2$, items 1 and 2 are also obvious. We proceed to show item 3 according to the four structure  types in Fig.~\ref{fig:structure-w-II}. For type II-1, since these is no descent before $w_{l_{s-1}}$ in $v$,
 then no patterns in $\{31425,32415,31524,32514\}$ can be formed by $w_{l_{s-1}}$ in $v$. Thus, item 3 follows for type II-1.
 For the other cases, we consider two situations.
 \begin{itemize}
 \item
 If $x \neq l_{s-2}$, then we have $l_{s-1}=l_{s-2}+1$.
 By the construction of $\alpha$, $w_{l_{s-2}}$
and  $w_{l_{s-1}}$ are adjacent in $v$.
It can be verified that any instance forming a pattern  in $\{31425,32415,31524,32514\}$ of $v$, which contains $w_{l_{s-1}}$ ($w_{l_{s-1}}$ must play the role of $4$ or $5$),
 will correspond to an instance forming a pattern in $\{31425,32415,31524,32514\}$ of $u$ by changing $w_{l_{s-1}}$ into $w_{l_{s-2}}$, a contradiction.
 \item
 If $x=l_{s-2}$, then $l_{s-1}>l_{s-2}+1$.
 For types II-2 and II-4, $w_{l_{s-1}}$
and  $w_{l_{s}}$ are adjacent in $v$.
Moreover, $w_{l_{s-1}}$ is the second largest element only smaller than $w_{l_s}$.
Hence, any pattern we  concern containing $w_{l_{s-1}}$ in $v$ can be changed to a concerned pattern containing $w_{l_s}$, a contradiction.
 For type II-3,  $w_{l_{s-1}}$
and  $w_{l_{s-1}+1}$ are adjacent in $v$.
Similarly, any pattern in $\{31425,32415,31524,32514\}$  containing $w_{l_{s-1}}$ in $v$ can be changed to a  pattern in $\{31425,32415,31524,32514\}$ containing $w_{l_{s-1}+1}$, a contradiction.
\end{itemize}
These facts confirm item 3 for types II-2, II-3 and II-4.

 \item When  $w_{l_{s-1}}>w_{r_2}$ with $s=2$, we find that $w_1$ can not be in any instance forming a pattern in $\{31425,32415,31524,32514\}$. Based on this, it is plain  to
 see that items 1, 2 and 3 hold.
 \end{itemize}
 The proof of the proposition is now completed.
 \end{proof}

Furthermore, we have the following  properties  of $\alpha$.
 \begin{lemma}\label{vrmax:stat}
 The mapping $\alpha$ preserves the triple of statistics $(\Ides,\Vrlmax,\Vlrmin)$.
 \end{lemma}
 \begin{proof}
 Given $w \in \W_n(31245,32145,31254,32154)$ and $v=\alpha(w)$, we aim to show
\begin{equation}\label{eq:stequ}
(\Ides,\Vrlmax,\Vlrmin)w=(\Ides,\Vrlmax,\Vlrmin)v.
\end{equation}
 When $w=\emptyset$, (\ref{eq:stequ}) certainly holds.
 Assume that it holds for the words of length $n-1$, we proceed to show that it also holds for $n$.
It can be  concluded that when we insert the largest or the second largest element in the process of $\alpha$, we never exchange their relative positions. Combining with the induction hypothesis, we see that $\Ides(w)=\Ides(v)$.

By the construction of $\alpha$, the insertion of the second largest element will have no effect on the right-to-left maxima, while the insertion of the largest element always keep relative positions with
its nearest right-to-left maximum if there is any.
Hence, by induction we deduce that $\Vrlmax(w)=\Vrlmax(v)$.

Based on the induction hypothesis,
the relation $\Vlrmin(w)=\Vlrmin(v)$ can be checked case by case according to the construction of $\alpha$ easily. This completes the proof of the lemma.
\end{proof}

By Lemma~\ref{vrmax:stat} and item~1 of Proposition~\ref{prop:alpha}, the mapping $\alpha$ preserves the pair of statistics $(\Vlrmax,\Vrlmax)$, which leads to the following observation.

\begin{Obser}\label{prop:relativeorderA}
 Given  a non-trivial word $w\in\W_n(31245,32145,31254,32154)$,  the relative order of $w_{l_{s-1}}, w_{l_s}$ and $w_{r_2}$ keeps unchanged after the mapping $\alpha$.
\end{Obser}

The above observation and the one below  are crucial in proving the bijectivity of $\alpha$.

\begin{Obser}\label{pro:consecutiveA}
Assume that $w$ is a non-trivial  $\{31245,32145,31254,32154\}$-avoiding word of different letters. We have
\begin{itemize}
  \item when $w_{l_{s-1}}<w_{r_2}$, $l_{s-1}+1=l_s$ if and only if
  $w_{l_{s-1}}$ and $w_{l_s}$  are  adjacent  in $\alpha(w)$.
  \item when $w_{l_{s-1}}>w_{r_2}$ and $s>2$, $l_{s-2}+1=l_{s-1}$ if and only if
  $w_{l_{s-2}}$ and $w_{l_{s-1}}$  are  adjacent  in $\alpha(w)$.
\end{itemize}
\end{Obser}

We need to prove the following proposition, from which Observation~\ref{pro:consecutiveA}   follows immediately.

\begin{proposition}\label{pro:desascA}
Suppose that $w$ is a  $(31245,32145,31254,32154)$-avoiding word of different letters.
If 
$w$ is of type II-4 (resp.~other cases),
then
\begin{itemize}
  \item when $w_{l_i} w_{l_i+1}$ is a descent for $1 \leq i \leq {s-2}$ (resp.~ $1 \leq i \leq {s-1}$),  we have $w_{l_i}$ and $w_{l_i+1}$  remain  adjacent  in $\alpha(w)$.
  \item when $w_{l_i} w_{l_i+1}$ is an ascent for $1 \leq i \leq {s-2}$ (resp.~ $1 \leq i \leq {s-2}$),  we have $w_{l_i}$ and $w_{l_i+1}$  remain  adjacent  in $\alpha(w)$.
\end{itemize}
\end{proposition}
\begin{proof}
We proceed to give the proof  by induction.
If $w=\emptyset$, Proposition \ref{pro:desascA} certainly holds.
Assume that it holds for the words  in $\W_{n-1}(31245,32145,31254,32154)$,
we aim to show that it  holds for such words with length $n$.
Cases
\begin{itemize}
\item when $\lrmax(w)=1$ and $\rlmax(w) \geq 1$
\item or when $\lrmax(w)>1$, $\rlmax(w) > 1$ and $l_s=s$
\item or when $w_{l_{s-1}}>w_{r_2}$ and $s=2$
\end{itemize}
can be verified easily. We focus on the remaining cases.

 If $\lrmax(w)>1$ and $\rlmax(w)=1$, then $\alpha(w)=\varphi(w_1 w_2 \cdots w_{n-1}) \max(w)$.
    By definition of the map $\varphi$, it is easy to check that
    if $w_{l_i} w_{l_i+1}$ is a descent (resp.~ascent) for $1 \leq i \leq {s-1}$ (resp.~$1 \leq i \leq s-2$) of $w_1 w_2 \cdots w_{n-1}$, then $w_{l_i}$ and $w_{l_i+1}$  remain  adjacent  in $\varphi(w)$, and hence in $\alpha(w)$. This completes the proof of this case.

 If $\lrmax(w)>1$, $\rlmax(w)>1$ and $l_s>s$, we distinguish  two cases.
\begin{itemize}
  \item When $w_{l_{s-1}}<w_{r_2}$, let $\tilde w=w_1 \cdots w_{l_{s}-1} w_{l_{s}+1} \cdots w_n$.
      Clearly, $w_{l_1}, w_{l_2}, \ldots, w_{l_{s-1}}$ are the first $s-1$ left to right maxima of $\tilde w$, and hence are also of $\alpha(\tilde w)$ by item $1$ in Proposition \ref{prop:alpha}.
      \begin{itemize}
        \item For type I-1, we have $\lrmax(\tilde w) \geq s$.
      By the induction hypothesis, elements $w_{l_i}$ and $w_{l_i+1}$ with
       $w_{l_i}>w_{l_i+1}$($1 \leq i \leq s-2$)
 or $w_{l_i}<w_{l_i+1}$($1 \leq i \leq s-2$)
      are adjacent in $\alpha(\tilde w)$. Notice that the inserting of $w_{l_s}$ brings no effect on this property and $w_{l_{s-1}}<w_{l_{s-1}+1}$. The proposition of this subcase follows.
        \item For type I-2, we claim that the word $\tilde w$ is not of type II-4. Otherwise,
      assume that $x=\max\{j \colon w_j> w_{j+1} \text{~~and~~} j<l_{s-1}\}$, we have $w_x>w_{l_{s-1}+1}$. This implies that $w_x w_{x+1} w_{l_{s-1}+1} w_{l_s} w_{r_2}$ forms a $31254$ or $32154$-pattern of $w$, a contradiction.
      The claim is verified. Then, by the induction hypothesis,
      elements $w_{l_i}$ and $w_{l_i+1}$ with $w_{l_i}>w_{l_i+1}$ ($1 \leq i \leq s-1$)
    or $w_{l_i}<w_{l_i+1}$ ($1 \leq i \leq s-2$)
      are adjacent in $\alpha(\tilde w)$, as well as in $\alpha(w)$. The proposition of this subcase is confirmed.
        \item For type I-3,  $\lrmax(\tilde w) \geq s+1$.
      Based on the induction hypothesis, we see that $w_{l_i}$ and $w_{l_i+1}$ with
      $w_{l_i}>w_{l_i+1}$($1 \leq i \leq s-1$)
     or $w_{l_i}<w_{l_i+1}$($1 \leq i \leq s-2$)
      are adjacent in $\alpha(\tilde w)$. Besides, the inserting of $w_{l_s}$ brings no effect on this property. The proposition of this subcase is verified.
      \end{itemize}
  \item When $w_{l_{s-1}}>w_{r_2}$ and $s>2$, let $\hat w=w_1 \cdots w_{l_{s-1}-1} w_{l_{s-1}+1} \cdots w_n$.
      Clearly, there are at least $s-1$ left to right maxima of
      $\hat w$, and the first $s-2$ ones are $w_{l_1}, \ldots, w_{l_{s-2}}$.
      \begin{itemize}
        \item  If $w$ is of type II-1 and $\hat w$ is of type II-4, then $w_{l_{s-2}}$ is at most the second largest element before $w_{l_s}$ in $\hat w$. This implies that in both cases, namely $\hat w$  is of type II-4 or not, we always have that
      elements $w_{l_{s-2}}$ and $w_{l_{s-1}+1}$,
       $w_{l_i}$ and  $w_{l_i+1}$ ($1 \leq i \leq s-3$)
      are adjacent in $\alpha(\hat w)$ by the induction hypothesis.
      Since $\alpha(w)$ is obtained  from $\alpha(\hat w)$ by inserting $w_{l_{s-1}}$ just after $w_{l_{s-2}}$, then $w_{l_{s-2}}$ and $w_{l_{s-1}}$, $w_{l_{s-1}}$ and $w_{l_{s-1}+1}$ are adjacent in $\alpha(w)$.
      The proposition for this subcase follows.
        \item If $w$ is of type II-2, we claim that $\hat w$ is not of type II-4. Otherwise, we deduce that $x=l_{s-2}$ and
       $w_y>w_{x+1}$ with $y=\max\{j \colon w_j> w_{j+1} \text{~~and~~} j<x\}$. This implies that $w_y w_{y+1} w_{x+1} w_{l_{s-1}} w_{l_s}$ forms a $31245$ or $32145$-pattern of $w$, a contradiction.
       Thus, the claim is verified. By the induction hypothesis,
      elements $w_{l_i}$ and $w_{l_i+1}$ with $w_{l_i}>w_{l_i+1}$ ($1 \leq i \leq s-2$)
    or $w_{l_i}<w_{l_i+1}$ ($1 \leq i \leq s-3$)
      are adjacent in $\alpha(\hat w)$.
      When $x = l_{s-2}$, then
      $w_{l_{s-1}}$ is inserted in $\alpha(\hat w)$ before $w_{l_{s}}$. This brings no effect on the properties above.
      When $x \neq l_{s-2}$,
      $w_{l_{s-1}}$ is inserted after $w_{l_{s-2}}$. Hence,
       $w_{l_{s-2}}$ and $w_{l_{s-1}}$ are adjacent in $\alpha(w)$.
     Combining all these facts, we complete the proof of this subcase.
        \item  If $w$ is of type II-3 and $x \neq l_{s-2}$,
        when $\hat w$ is of type II-4, $w_{l_{s-2}}$ is at most the
        second element before $w_{l_s}$ in
         $\hat w$. Thus, no matter
         $\hat w$ is of type II-4 or not, we have
         $w_{l_{i}}$ and $w_{l_{i}+1}$ with $1 \leq i \leq s-3$, $w_{l_{s-2}}$ and $w_{l_{s-1}+1}$ are adjacent in $\alpha(\hat w)$ by induction.
         Since $\alpha(w)$ is obtained by inserting $w_{l_{s-1}}$ after $w_{l_{s-2}}$,
         then $w_{l_{s-2}}$ and  $w_{l_{s-1}}$, $w_{l_{s-1}}$ and  $w_{l_{s-1}+1}$ are adjacent in $\alpha(w)$ respectively. The proposition for this situation holds.
        If $w$ is of type II-3 and $x = l_{s-2}$,
        then $w_{l_{s-2}}<w_{l_{s-1}+1}$ and    $\lrmax(\hat w) \geq s$.
        Both situations that $\hat w$ is of type II-4 or not,
        we always have
       $w_{l_i}$ and $w_{l_i+1}$ are adjacent
        with $1 \leq i \leq s-2$. Besides, $\alpha(w)$ is obtained by inserting  $w_{l_{s-1}}$ before $w_{l_{s-1}+1}$ in $\alpha(\hat w)$. Thus, $w_{l_{s-1}}$ and
         $w_{l_{s-1}+1}$ remain adjacent in $\alpha(w)$.
        The proposition of this subcase is verified.
  \item If $w$ is of type II-4 and $x=l_{s-2}$,
        then $\hat w$ is not of type II-4.
        Elements $w_{l_{i}}$ and $w_{l_{i}+1}$
            ($1 \leq i \leq s-2$) are adjacent in $\alpha(\hat w)$ by induction.
     As $\alpha(w)$ is obtained by inserting $w_{l_{s-1}}$
     before $w_{l_{s}}$ in $\alpha(\hat w)$,
     elements $w_{l_{i}}$ and $w_{l_{i}+1}$
            ($1 \leq i \leq s-2$) remain adjacent in $\alpha(w)$, as desired.
     If $w$ is of type II-4 and $x \neq l_{s-2}$,
     then  $w_{l_{i}}$ and $w_{l_{i}+1}$
            ($1 \leq i \leq s-3$) are adjacent in $\alpha(\hat w)$ by induction.
      Notice that $\alpha(w)$ is obtained by inserting $w_{l_{s-1}}$ after $w_{l_{s-2}}$. It follows that  $w_{l_{s-2}}$ and $w_{l_{s-1}}$ are adjacent in $\alpha(w)$. This verifies the last subcase.
      \end{itemize}
\end{itemize}

The proof of   this proposition is now complete.
\end{proof}

 In order to show that $\alpha$ is a bijection, we introduce its inverse $\beta$.

 {\bf The construction of $\beta$.} Given a word $v=v_1 v_2 \cdots v_n \in \W_n(31425,32415,31524,32514)$,
assume that $\Vlrmax(v)=\{v_{a_1}, v_{a_2},\ldots, v_{a_h} \}$ and
 $\Vrlmax(v)=\{v_{b_1}, v_{b_2},\ldots, v_{b_g} \}$ with $1=a_1<\cdots<a_h=b_1< \cdots <b_g=n$.
We construct $\beta(v)$ through the following cases:
\begin{itemize}
  \item If $n=0$, then define $\beta(\emptyset)=\emptyset$.

  \item If $\lrmax(v)=1$ and $\rlmax(v) \geq 1$, then
        $\beta(v)=\max(v) \beta(v_2 \cdots v_n)$.

  \item If $\lrmax(v)>1$ and $\rlmax(v) = 1$, then
        $\beta(v)= \psi(v_1 \cdots v_{n-1})\max(v)$.

  \item If $\lrmax(v)>1$ and $\rlmax(v)>1$, then we consider the following cases.
      \begin{itemize}
        \item[(a)] If $a_h=h$,
        then define $\beta(v)$ to be  the word obtained by inserting $\max(v)$ into $\beta(v_1 \cdots v_{a_h-1} v_{a_h+1} \cdots v_n)$ at position $a_h$.

        \item[(b)] If $a_h>h$,
        we consider three cases.
\begin{itemize}
       \item [(b1)] If $v_{a_{h-1}}<v_{b_2}$, then let $e=\beta(v_1 \cdots v_{a_h-1}v_{a_h+1} \cdots v_n).$
           When $a_{h-1}+1=a_h$,
           we construct $\beta(v)$ by inserting $v_{a_{h}}$ into $e$ just after the $(h-1)$-th left-to-right maximum of $e$.
           When $a_{h-1}+1<a_h$,
           we construct $\beta(v)$ by inserting $v_{a_{h}}$ into $e$  just after the element closely following the $(h-1)$-th left-to-right maximum of $e$.

           \item[(b2)] If $v_{a_{h-1}}>v_{b_2}$ and $h>2$,
           then denote  $e=\beta(v_1 \cdots v_{a_{h-1}-1}v_{a_{h-1}+1} \cdots v_n)$.
           When $a_{h-2}+1=a_{h-1}$,
           we construct $\beta(v)$ by inserting $v_{a_{h-1}}$ into $e$ just after the $(h-2)$-th left-to-right maximum of $e$.
          When $a_{h-2}+1<a_{h-1}$,
            insert $v_{a_{h-1}}$ into $e$ just after the
            element closely following the $(h-2)$-th left-to-right maximum of $e$ and we obtain $\beta(v)$.

          \item[(b3)] If $v_{a_{h-1}}>v_{b_2}$ and $h=2$,
           then $\beta(v)$ can be obtained by inserting
           $v_1$ at the beginning of $\beta(v_2 \cdots v_n)$.
           \end{itemize}
      \end{itemize}

\end{itemize}

 To show that $\beta$ is well defined and further the inverse of $\alpha$, we need the following proposition.
\begin{proposition}
Given $v$ in $\W_n(31425,32415,31524,32514)$,
we have
\begin{itemize}\label{prop:beta}
  \item[1.] $\Vlrmax(\beta(v))=\{v_{a_1}, v_{a_2},\ldots,$ $ v_{a_h} \}$.

  \item[2.]  $\beta(v)$ avoids patterns in $\{31245,32145,31254,32154\}$.
  \item[3.] If $v_{a_i} v_{a_i+1}$ is a descent (resp.~ascent) for $1 \leq i \leq {h-1}$ (resp.~$1 \leq i \leq h-2$), then $v_{a_i}$ and $v_{a_i+1}$  remain  adjacent  in $\beta(v)$.
\end{itemize}
Consequently,  the map $\beta$ is  well defined.
\end{proposition}

\begin{proof}
We proceed to give the proof  by induction.
If $v=\emptyset$, Proposition \ref{prop:beta} certainly holds.
Assume that it holds for the words  in $\W_{n-1}\{31425,32415,31524,32514\}$,
we aim to show that it  holds for such words with length $n$.
We just present the proof of the cases below, while  the remaining three cases, namely
\begin{itemize}
\item when $\lrmax(v)=1$ and $\rlmax(v) \geq 1$
\item or when $\lrmax(v)>1$, $\rlmax(v) > 1$ and $a_h=h$
\item or when $v_{a_{h-1}}>v_{b_2}$ and $h=2$,
\end{itemize}
can be verified easily.

If $\lrmax(v)>1$ and $\rlmax(v) = 1$, then
        $\beta(v)= \psi(v_1 \cdots v_{n-1})\max(v)$.
        Item 1 is obvious by combining  $\psi=\varphi^{-1}$ and item 1 in Proposition \ref{prop:var}. Since $v$ avoids $\{31425,32415,31524,$ $32514\}$, we have $v_1 \cdots v_{n-1}$
        avoids $\{3142,3241\}$. It follows that $\psi(v_1 \cdots v_{n-1}) $ avoids $\{3214,$ $3124\}$, and hence $\beta(v)$
        avoids $\{31245,32145,31254,$ $32154\}$.
        As for item 3, it suffices to check that
        $y_i y_{i+1}$ remain consecutive in $\psi(y)$ for the
        left-to-right maximum $y_i$ with $y_i \neq \max(y)$.
        This is obvious in view of the construction of $\psi$ 
        and we obtain item 3.

If $\lrmax(v)>1$, $\rlmax(v) > 1$ and $a_h>h$, then we consider the following two cases.
\begin{itemize}
  \item If $v_{a_{h-1}}<v_{b_2}$, then let $e=\beta(v_1 \cdots v_{a_h-1}v_{a_h+1} \cdots v_n).$
      Clearly, $v_{a_1}, \ldots,$ $ v_{a_{h-1}}, v_{b_2}$ are left-to-right maxima of the word $v_1 \cdots v_{a_h-1}v_{a_h+1} \cdots v_n$.
      By the induction hypothesis, we have $\{v_{a_1}, \cdots,$ $ v_{a_{h-1}}, v_{b_2}\} \subseteq \Vlrmax(e)$. We distinguish the following two subcases.
      \begin{itemize}
     \item When $a_{h-1}+1=a_h$,
      $v_{a_{h}}$ is inserted  just after $v_{a_{h-1}}$ in $e$ and so $\Vlrmax(\beta(v))=\{v_{a_1}, \cdots,$ $ v_{a_{h-1}}, v_{a_h}\}$. For item 2, assume to the contrary that  $w_{x_1} w_{x_2} w_{x_3} w_{x_4} w_{x_5}$ in $\beta(v)$ forms a pattern in $\{31245,32145,31254,32154\}$.
      By  the induction hypothesis that $e$ avoids such patterns,  we have $w_{x_4}=v_{a_h}$ or $w_{x_5}=v_{a_h}$.
      If $w_{x_4}=v_{a_h}$, then $w_{x_1} w_{x_2} w_{x_3} v_{a_{h-1}} w_{x_5}$ forms a pattern in  $\{31245,32145,$ $31254,32154\}$ of $e$, a contradiction.
      If $w_{x_5}=v_{a_h}$, then $w_{x_1} w_{x_2} w_{x_3} w_{x_4} v_{b_2}$ forms a $31245$ or $32145$-pattern of $e$, a contradiction. Both cases  indicate that $\beta(v)$ avoids
      $\{31245,32145,31254,32154\}$.
      For item 3, noticing that $v_{a_h}$ is inserted after $v_{a_{h-1}}$ in $\beta(v)$, then it is obvious in view of the induction hypothesis.

    \item  When $a_{h-1}+1<a_h$, by induction we see that $v_{a_{h-1}}$ and $v_{a_{h-1}+1}$ remain adjacent in $e$, and hence
       $\beta(v)$ is obtained from $e$ by inserting $v_{a_{h}}$  just after $v_{a_{h-1}+1}$. Based on the induction hypothesis,
      it is plain  to see that  $\Vlrmax(\beta(v))=\{v_{a_1}, \cdots,$ $ v_{a_{h-1}}, v_{a_h}\}$.
      To prove item 2, we assume to the contrary that $\beta(v)$ contains $w_{x_1} w_{x_2} w_{x_3} w_{x_4} w_{x_5}$ a pattern  in $\{31245,32145,31254,32154\}$.
      Since $e$ avoids such patterns by  the induction hypothesis,  we have $w_{x_4}=v_{a_h}$ or $w_{x_5}=v_{a_h}$.
      If $w_{x_4}=v_{a_h}$, we note that $w_{x_3} \neq v_{a_{h-1}+1}$.
       Otherwise, consider the greatest $a_i$ with $a_i<a_{h-1}$ such that $v_{a_i}>v_{a_i+1}$. By the induction hypothesis on $e$, we must have $v_{a_i}\geq w_{x_1}$ in view of item~3. Thus,
      the subsequence $ v_{a_i}v_{a_i+1} v_{a_{h-1}} v_{a_{h-1}+1} v_{a_{h}}$ will
       form a $31425$ or $32415$-pattern of $v$, which contradicts with $v \in \W_n(31425,32415,31524,32514)$.
       It follows that $w_{x_3}$ must appear before $v_{a_{h-1}}$ in $\beta(v)$. This implies that $w_{x_1} w_{x_2} w_{x_3}  v_{a_{h-1}} w_{x_5}$ is a pattern in $\{31245,32145,31254,32154\}$ of $e$,
       a contradiction.
       If $w_{x_5}=v_{a_h}$, then $w_{x_1} w_{x_2} w_{x_3} w_{x_4} v_{b_2}$ forms a $31245$ or $32145$  pattern in $e$, a contradiction. Both cases imply item 2.
      Item 3 is obvious in view of the fact that
    the insertion of $v_{a_h}$ is after $v_{a_{h-1}+1}$ and  the induction hypothesis.
\end{itemize}
   \item If $v_{a_{h-1}}>v_{b_2}$ and $h>2$, then let $e=\beta(v_1 \cdots v_{a_{h-1}-1}v_{a_{h-1}+1}\cdots v_n).$
       It is clear that
       $\{v_{a_1}, \ldots,$ $ v_{a_{h-2}}, v_{a_h}\} \subseteq \Vlrmax(e)$ by the induction hypothesis. We further distinguish the following two subcases.
       \begin{itemize}
     \item When $a_{h-2}+1=a_{h-1}$,
      $v_{a_{h-1}}$ is inserted  just after $v_{a_{h-2}}$ in $e$. Similar discussions as the case when $v_{a_{h-1}}<v_{b_2}$ and $a_{h-1}+1=a_h$ can be used to show
      items~1 and 2 in this case.
      To prove item 3,  it suffices to show that $v_{a_{h-1}}$ and $v_{a_{h-1}+1}$ are adjacent in $\beta(v)$ if $a_{h-1}+1<a_{h}$. Whether  $v_{a_{h-2}}>v_{a_{h-1}+1}$ or $v_{a_{h-2}}<v_{a_{h-1}+1}$,
      we have $v_{a_{h-2}}$ and $v_{a_{h-1}+1}$ are adjacent in $e$ by
      induction hypothesis.
      Since $\beta(v)$ is obtained from $e$ by inserting $v_{a_{h-1}}$  just after $v_{a_{h-2}}$,
      item 3 for this case follows.

     \item When $a_{h-2}+1<a_{h-1}$,  by induction we see that $v_{a_{h-2}}$ and $v_{a_{h-2}+1}$ remain adjacent in $e$, and hence in $\beta(v)$. As
      $\beta(v)$ is obtained from $e$ by inserting $v_{a_{h-1}}$  just after $v_{a_{h-2}+1}$, we have $\Vlrmax(\beta(v))=\{v_{a_1}, \cdots,$ $ v_{a_{h-1}}, v_{a_h}\}$. To prove item 2,
      we assume to the contrary that  $w_{x_1} w_{x_2} w_{x_3} w_{x_4} w_{x_5}$ is a pattern in $\{31245,32145,31254,32154\}$ of $\beta(v)$. Since $e$ avoids such patterns by  the induction hypothesis, we deduce that
       $w_{x_4}=v_{a_{h-1}}$ or $w_{x_5}=v_{a_{h-1}}$.
      If $w_{x_4}=v_{a_{h-1}}$,  If $w_{x_4}=v_{a_{h-1}}$,  we may deduce that $w_{x_3}\neq v_{a_{h-2}+1}$ using the same discussions as in the case $v_{a_{h-1}}<v_{b_2}$ and $a_{h-1}+1<a_h$. Thus,
      $w_{x_3}$ appears before $v_{a_{h-2}}$ in $e$.
      It follows that
      $w_{x_1} w_{x_2} w_{x_3} v_{a_{h-2}} w_{x_5}$ will form a pattern in  $\{31245,32145,$ $31254,32154\}$ of $e$, a contradiction.
      If $w_{x_5}=v_{a_{h-1}}$, then $w_{x_1} w_{x_2} w_{x_3} w_{x_4} v_{a_h}$ forms a $31245$ or $32145$-pattern of $e$, a contradiction. Both cases  indicate that $\beta(v)$ avoids
      $\{31245,32145,31254,32154\}$.

      For item 3,  we consider two cases.
       If $a_{h-1}+1=a_h$, the insertion of  $v_{a_{h-1}}$ after
       $v_{a_{h-2}+1}$ will bring no effect on descents (resp.~ascents)
       beginning with a left-to-right maximum less than $v_{a_{h-1}}$ (resp.~$v_{a_{h-2}}$) in $e$. By the induction  hypothesis, item 3 follows.
       If $a_{h-1}+1<a_h$, then  $v_{a_{h-1}+1}>v_{a_{h-2}}$.
       Otherwise, $v_{a_{h-2}} v_{a_{h-2}+1} v_{a_{h-1}} v_{a_{h-1}+1} v_{a_{h}}$ will form a $31425$ or $32415$-pattern of $v$.
       We claim that $v_{a_{h-2}+1}$ and $v_{a_{h-1}+1}$
       are adjacent in $e$. Assume to the contrary, if there is an element $x$ between them, then $v_{a_{h-2}} v_{a_{h-2}+1} x  v_{a_{h-1}+1} v_{a_h}$ forms a $31245$ or $32145$-pattern of $e$. This contradicts with the induction hypothesis. Thus, the claim is verified. It follows that $v_{a_{h-1}}v_{a_{h-1}+1}$
       remains a descent of $\beta(v)$.
       Finally, it is plain to  see that
       the insertion of  $v_{a_{h-1}}$ after
       $v_{a_{h-2}+1}$ brings no effect on descents (resp.~ascents)
       beginning with a left-to-right maximum less than $v_{a_{h-1}}$ (resp.~$v_{a_{h-2}}$) in $e$. We compete the proof of item 3.
\end{itemize}
  \end{itemize}

The proof of this proposition is completed.
\end{proof}

To show that $\beta$ and $\alpha$ are inverses of each other, we need further explore the property of $\beta$ and analyze the structure of the words in $\W_n(31425,32415,31524,32514)$.
 Given such a word $v$, we focus on the  non-trivial cases when
 $\lrmax(v)>1$, $\rlmax(v)>1$ and $a_h>h$ (i.e., case~(b) in the construction of $\beta$). The next observation for $\beta$ is parallel to Observation~\ref{prop:relativeorderA}  for $\alpha$.

\begin{Obser}\label{prop:relativeorderB}
Given a non-trivial word $v\in\W_n(31425,32415,31524,32514)$,
the relative order of $v_{a_{h-1}}, v_{a_h}$ and $v_{b_2}$ remains the same after the mapping $\beta$.
\end{Obser}

\begin{proof}
This observation follows immediately from item $1$ in Proposition~\ref{prop:beta} when
 $v_{a_{h-1}}<v_{b_2}$.
 We proceed to show that it
 also holds for $v\in\W_n(31425,32415,31524,32514)$ with $v_{a_{h-1}}>v_{b_2}$.

When $h>2$, let $e=\alpha(v_1\cdots v_{a_{h-1}-1} v_{a_{h-1}+1}\cdots v_n)$. Notice that there are at least $h-2$ left to right maxima before
$v_{a_h}$ in $v_1\cdots v_{a_{h-1}-1} v_{a_{h-1}+1}\cdots v_n$. Through
item $1$ in Proposition~\ref{prop:beta}, we see that $v_{a_h}$ is at least the $(h-1)$-th left to right maxima of $e$.
If $a_{{h-2}}+1=a_{h-1}$, then $v_{a_{h-1}}$ is inserted in $e$
just after $v_{a_{h-2}}$. If $a_{{h-2}}+1<a_{h-1}$, then
$v_{a_{h-1}}$ is inserted in $e$
just after the element closely following
$v_{a_{h-2}}$. In view of item 3 in Proposition~\ref{prop:beta},
we see that the element can  not be $v_{a_h}$.
In each case, $v_{a_{h-1}}$ is inserted before $v_{a_h}$.
By induction, we see that $v_{a_h}$ is always to the left of $v_{b_2}$.
The observation for this case is verified.

 When $h=2$,  $v_{a_{h-1}}=v_1$ is inserted at the beginning of $\beta(v_2 \cdots v_n)$.
Further, by induction, $v_{a_h}$ is to the left of $v_{b_2}$ in $\beta(v_2 \cdots v_n)$. Combining these two properties, the observation for this case is verified and  the proof is complete.
\end{proof}


The next two lemmas analyze the structures of non-trival
$(31425,32415,31524,32514)$-avoiding words of different letters. Their proofs are straightforward using discussions similar to that in Lemma~\ref{prop:<}, which are omited.
\begin{lemma}\label{prop:<B}
Assume that $v$ is a non-trivial $(31425,32415,31524,32514)$-avoiding word of different letters with $v_{a_{h-1}}<v_{b_2}$ and $a_h>h$,  there are totally three types:
\begin{itemize}
  \item[A-1.]  $a_h=a_{h-1}+1$.
  \item[A-2.]  $a_h>a_{h-1}+1$ and $b_2 = a_h+1$.
  \item[A-3.]  $a_h>a_{h-1}+1$, $b_2>a_h+1$ and $v_j>v_{a_{h-1}}$ for $a_h<j<b_2$.
 \end{itemize}
\end{lemma}

Based on Lemma~\ref{prop:<B}, we give the corresponding graphical descriptions in Fig.~\ref{fig:structure-v-A}.

\begin{figure}[!htbp]
\begin{center}
\begin{tikzpicture}[line width=0.5pt,scale=0.27]
\coordinate (O) at (0,0);
\coordinate (O1) at (-18,0);
\coordinate (O2) at (-37,0);
\fill[black!100] (O)circle(1ex)++(1,-4) circle(1ex)++(3,9.5) circle(1ex)++(1,-3) circle(1ex)++(2,2) circle(1ex);
\draw[pattern=north west lines] (5,0) rectangle (7,4.5);
\draw[pattern=north west lines] (1,-7) rectangle (4,0);

\path (-0.5,0.9)  node {\tiny{$v_{a_{h-1}}$}}
++(-0.7,-4) node {\tiny{$v_{a_{h-1}+1}$}}
++(4.8,9.5) node {\tiny{$v_{a_{h}}$}}
++(3.5,-1) node {\tiny{$v_{b_2}$}}
++(-4,-2.5) node {\tiny{$v_{a_{h}+1}$}}
++(0,-12) node {\small{Type A-3}};

\fill[black!100] (O1)circle(1ex)++(1,-4) circle(1ex)++(4,9.5) circle(1ex)++(1,-1) circle(1ex);
\draw[pattern=north west lines] (-17,-7) rectangle (-13,0);

\path (-18.5,0.9)  node {\tiny{$v_{a_{h-1}}$}}
++(-0.6,-4) node {\tiny{$v_{a_{h-1}+1}$}}
++(6,9.5) node {\tiny{$v_{a_{h}}$}}
++(2,-1.5) node {\tiny{$v_{b_2}$}}
++(-4,-14) node {\small{Type A-2}};

\fill[black!100] (O2)++(1,0)circle(1ex)++(1,5.5) circle(1ex)++(1,-3) ++(5,2) circle(1ex);

\draw[pattern=north west lines] (-35,-7) rectangle (-29,4.5);

\path (-36.5,0.9)  node {\tiny{$v_{a_{h-1}}$}}
++(0.4,-4)
++(2,9.5) node {\tiny{$v_{a_{h}}=v_{a_{h-1}+1}$}}
++(5.5,-1) node {\tiny{$v_{b_2}$}}
++(-3,-14.5) node {\small{Type A-1}};

\end{tikzpicture}
\caption{The structure of a non-trivial $(31425,32415,31524,32514)$-avoiding word with $v_{a_{h-1}}<v_{b_2}$ .}
\label{fig:structure-v-A}
\end{center}
\end{figure}
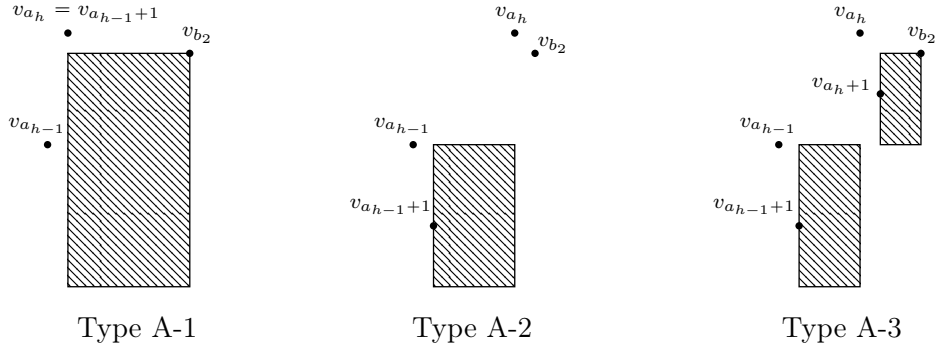

\begin{lemma}\label{prop:>B}
Assume that $v$ is a non-trivial $\{31425,32415,31524,32514\}$-avoiding word of different letters with $v_{a_{h-1}}>v_{b_2}$, $a_h>h$ and $h>2$, let $x=\max(\{ j \colon v_j> v_{j+1} \text{~and~} j<a_{h-1}\}\cup\{0\})$.
 There are totally three types:
\begin{itemize}
  \item[B-1.]  $x=0$.
  \item[B-2.]  $x \neq 0$ and $a_h=a_{h-1}+1$.
  Further, if $v_{b_2}>v_x$, then $v_{b_2}>v_j>v_x$ for $b_2>j>a_h$.
  \item[B-3.]   $x \neq 0$, $a_h>a_{h-1}+1$
  and $v_{a_{h-1}}>v_j>v_x$ with $a_{h}>j>a_{h-1}$. Further, if $v_{b_2}>v_x$, then $v_{b_2}>v_j>v_x$ for $b_2>j>a_h$.
 \end{itemize}
\end{lemma}

Based on Lemma~\ref{prop:>B}, its corresponding graphical description is given in Fig.~\ref{fig:structure-v-B}, where the gray boxes represent
consecutive increasing sequences which might be empty.

\begin{figure}[!htbp]
\begin{center}
\begin{tikzpicture}[line width=0.5pt,scale=0.27]
\coordinate (O) at (0,0);
\coordinate (O1) at (18,0);
\coordinate (O2) at (37,0);
\coordinate (O3) at (41,0);
\draw[pattern=north west lines] (7.5,-7) rectangle (10.5,1.5);
\draw[pattern=north west lines] (5,-7) rectangle (7.5,5.5);
\draw [fill=gray!30,ultra thick,gray!30] (1,-7) rectangle (4,5.5);

\fill[black!100] (O)++(4,5.5) circle(1ex)++(1,-5) circle(1ex)
++(2.5,6) circle(1ex)++(3,-5) circle(1ex);

\path (2.7,6.2)  node {\tiny{$v_{a_{h-1}}$}}
++(0.2,-5.5)node {\tiny{$v_{a_{h-1}+1}$}}
++(5.3,6.5) node {\tiny{$v_{a_h}$}}
++(3,-5) node {\tiny{$v_{b_2}$}}
++(-5.5,-10.5) node {\small{Type B-1}};

\draw[pattern=north west lines] (26,0) rectangle (28,1.5);
\draw[pattern=north west lines] (19,-7) rectangle (21,0);
\draw [fill=gray!30,ultra thick,gray!30] (21,0) rectangle (24,5.5);

\fill[black!100] (O1)circle(1ex)++(1,-3) circle(1ex)++(5,8.5) circle(1ex)++(1,1) circle(1ex)++(3,-5) circle(1ex);

\path (22.7,6.2)  node {\tiny{$v_{a_{h-1}}$}}
++(4,1) node {\tiny{$v_{a_h}=v_{a_{h-1}+1}$}}
++(1.6,-5) node {\tiny{$v_{b_2}$}}
++(-4,-10.5) node {\small{Type B-2(1)}};

\path (17.8,0.6)  node {\tiny{$v_{x}$}}
++(-0.4,-2.8) node {\tiny{$v_{x+1}$}};

\draw[pattern=north west lines] (45,-7) rectangle (47,-1.5);
\draw[pattern=north west lines] (38,-7) rectangle (40,0);
\draw [fill=gray!30,ultra thick,gray!30] (40,0) rectangle (43,5.5);

\fill[black!100] (O2)circle(1ex)++(1,-3) circle(1ex)++(5,8.5) circle(1ex)++(1,1) circle(1ex)++(3,-8) circle(1ex);

\path (41.7,6.2)  node {\tiny{$v_{a_{h-1}}$}}
++(4,1) node {\tiny{$v_{a_h}=v_{a_{h-1}+1}$}}
++(1.6,-8) node {\tiny{$v_{b_2}$}}
++(-4,-7.5) node {\small{Type B-2(2)}};

\path (36.8,0.6)  node {\tiny{$v_{x}$}}
++(-0.4,-2.8) node {\tiny{$v_{x+1}$}};

\draw[pattern=north west lines] (18,-20) rectangle (20,-18.5);
\draw[pattern=north west lines] (9,-27) rectangle (11,-20);
\draw[pattern=north west lines] (15,-15.5) rectangle (17,-20);
\draw [fill=gray!30,ultra thick,gray!30] (11,-20) rectangle (14,-15.5);

\fill[black!100] (O1)++(-10,-20)circle(1ex)++(1,-3) circle(1ex)++(5,7.5) circle(1ex)++(1,-2) circle(1ex)++(2,3) circle(1ex)++(3,-4) circle(1ex);

\path (12.4,-15)  node {\tiny{$v_{a_{h-1}}$}}
++(4,1.2) node {\tiny{$v_{a_h}$}}
++(4.5,-4.5) node {\tiny{$v_{b_2}$}}
++(-6.5,-11) node {\small{Type B-3(1)}}
++(-1.4,11.8) node {\tiny{$v_{a_{h-1}+1}$}};

\path (7.8,-20.6)  node {\tiny{$v_{x}$}}
++(-0.4,-2.8) node {\tiny{$v_{x+1}$}};

\draw[pattern=north west lines] (38,-27) rectangle (40,-21.5);
\draw[pattern=north west lines] (29,-27) rectangle (31,-20);
\draw[pattern=north west lines] (35,-15.5) rectangle (37,-20);
\draw [fill=gray!30,ultra thick,gray!30] (31,-20) rectangle (34,-15.5);

\fill[black!100] (O2)++(-9,-20)circle(1ex)++(1,-3) circle(1ex)++(5,7.5) circle(1ex)++(1,-2) circle(1ex)++(2,3) circle(1ex)++(3,-7) circle(1ex);

\path (32.4,-15)  node {\tiny{$v_{a_{h-1}}$}}
++(4,1.2) node {\tiny{$v_{a_h}$}}
++(4.5,-7.5) node {\tiny{$v_{b_2}$}}
++(-6.5,-8) node {\small{Type B-3(2)}}
++(-1.4,11.8) node {\tiny{$v_{a_{h-1}+1}$}};

\path (27.8,-20.6)  node {\tiny{$v_{x}$}}
++(-0.4,-2.8) node {\tiny{$v_{x+1}$}};
\end{tikzpicture}
\caption{The structure of a non-trivial $(31425,32415,31524,32514)$-avoiding word with $v_{a_{h-1}}>v_{b_2}$ and $h>2$.}
\label{fig:structure-v-B}
\end{center}
\end{figure}
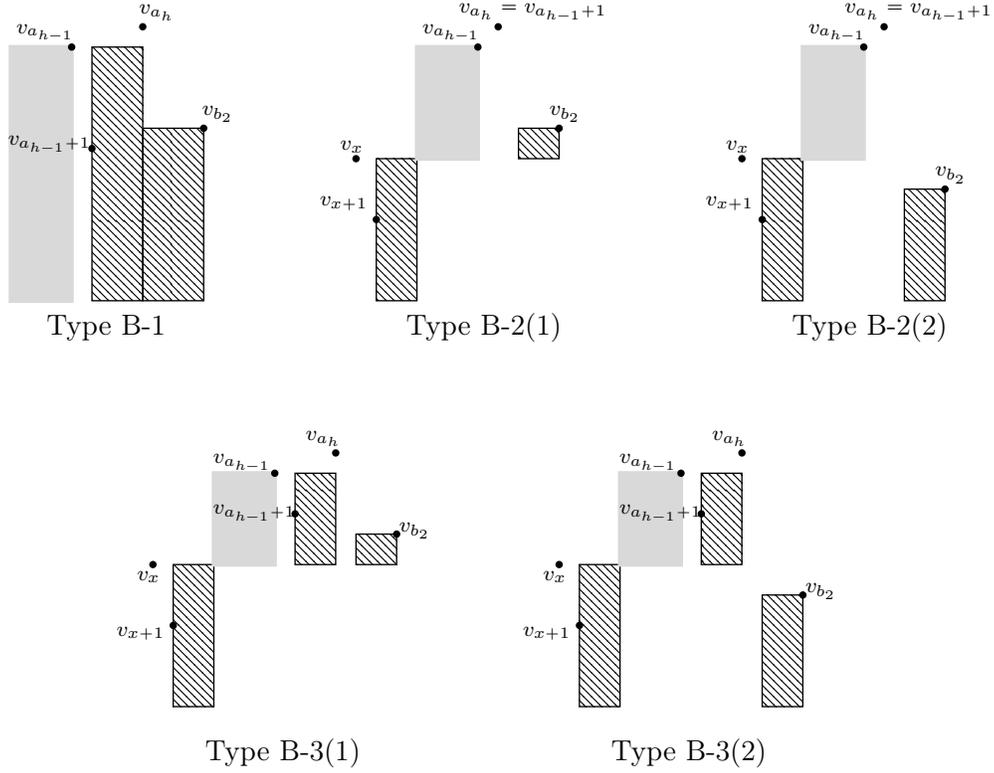

Now, we are ready to show that $\alpha$ and $\beta$ are inverses of each other.
\begin{proposition}\label{prop:bij}
We have $\beta \circ \alpha=I$ and $\alpha \circ \beta=I$.
\end{proposition}

\begin{proof}
To prove that $\beta \circ \alpha=I$, it suffices to show that
\begin{equation}\label{eq:alphabeta}
\beta (\alpha(w))=w
\end{equation}
with $w$ avoiding $\{31245, 32145,31254,32154\}$.
For the cases when $\lrmax(w)=1$, $\rlmax(w)=1$ or $l_s=s$, they can be readily  checked. We now focus on the non-trivial cases
when $\lrmax(w)>1$, $\rlmax(w)>1$ and $l_s>s$.
Assume that (\ref{eq:alphabeta}) holds for all words avoiding $\{31245,$ $ 32145,31254,32154\}$ with length less than $n$.
We wish to show that it is also valid for those of length $n$.

If $w_{l_{s-1}}< w_{r_2}$, let $\tilde w=w_1 \cdots w_{l_{s}-1} w_{l_{s}+1}\cdots w_n$. We consider the following three cases.
\begin{itemize}
  \item When $w$ is  of type I-1, namely $l_{s-1}+1=l_s$, $\alpha(w)$ is obtained by inserting
      $w_{l_s}$ just after $w_{l_{s-1}}$ in $\alpha(\tilde w)$.  It follows from
      item 1 in Proposition~\ref{prop:alpha} and Observation~\ref{prop:relativeorderA} that $\alpha(w)$ is
      a word of type A-1 in $\W_n(31425,32415,31524,32514)$. Consequently, inserting $w_{l_s}$ after $w_{l_{s-1}}$ in $\beta(\alpha(\tilde w))$, we obtain $\beta(\alpha(w))$.
      As $\beta(\alpha(\tilde w))$ $=\tilde w$ by induction, we see that $\beta(\alpha(w))=w$ in this case.
  \item  When $w$ is of type I-2 (resp.~I-3),
      $\alpha(w)$ is obtained by inserting
      $w_{l_s}$ just before $w_{r_{2}}$(resp.~$w_{l_s+1}$) in $\alpha(\tilde w)$. Based on Observation~\ref{pro:consecutiveA}, $w_{l_{s-1}}$ and $w_{l_s}$ are not
      adjacent in $\alpha(w)$, while they remain the $(s-1)$-th
      and the $s$-th left to right maxima of $\alpha(w)$, respectively. In view of Observation~\ref{prop:relativeorderA},
        $\alpha(w)$ is
      a word of type A-2 (resp.~A-3) in $\W_n(31425,32415,31524,32514)$. Consequently, inserting $w_{l_s}$ after the element closely following $w_{l_{s-1}}$ in $\beta(\alpha(\tilde w))=\tilde w$, we obtain $\beta(\alpha(w))$ which is equal to $w$.
\end{itemize}

If $w_{l_{s-1}}> w_{r_2}$ and $s>2$, let $\hat w=w_1 \cdots w_{l_{s-1}-1} w_{l_{s-1}+1}\cdots w_n$. We consider the following four cases.
\begin{itemize}
  \item When $w$ is of type II-1, namely $l_{s-2}+1=l_{s-1}$, $\alpha(w)$ is obtained by inserting
      $w_{l_{s-1}}$ just after $w_{l_{s-2}}$ in $\alpha(\hat w)$.  In view of
      item 1 in Proposition~\ref{prop:alpha} and Observation~\ref{prop:relativeorderA}, it follows that $\alpha(w)$ is a word of type B-1 in $\W_n(31425,32415,31524,32514)$. Consequently, inserting $w_{l_{s-1}}$ after $w_{l_{s-2}}$ in $\beta(\alpha(\hat w))$, we obtain $\beta(\alpha(w))$.
      As $\beta(\alpha(\hat w))=\hat w$ by induction, we see that $\beta(\alpha(w))=w$ in this case.
  \item  If $w$ is of type II-2 (resp.~II-3, II-4) and $x \neq l_{s-2}$,
      $\alpha(w)$ is obtained by inserting
      $w_{l_{s-1}}$ just after $w_{l_{s-2}}$
      in $\alpha(\hat w)$. Similarly as that of type II-1, we may prove that $\beta(\alpha(w))=w$.    If $w$ is of type II-2 (resp.~II-3, II-4) and $x = l_{s-2}$, then $\alpha(w)$ can be obtained by inserting $w_{l_{s-1}}$ just before
      $w_{l_s}$ (resp.~$w_{l_{s-1}+1}$, $w_{l_s}$). By item 1 in Proposition~\ref{prop:alpha} and Observation~\ref{prop:relativeorderA}, $\alpha(w)$ is of type B-2 (resp.~B-3, B-2) in
      $\W_n(31425,32415,31524,32514)$.
      By Observation~\ref{pro:consecutiveA}, $w_{l_{s-2}}$ and $w_{l_{s-1}}$ are not adjacency in $\alpha(w)$.
       Thus, $\beta(\alpha(w))$
      is obtained by inserting $w_{l_{s-1}}$ just after the element closely following $w_{l_{s-2}}$ in $\beta(\alpha(\hat w))$. By the induction hypothesis,
      $\beta(\alpha(\hat w))
     =\hat w$ and hence we have $\beta(\alpha(w))=w$ for this subcase.
\end{itemize}

The case when $w_{l_{s-1}}> w_{r_2}$ and $s=2$ can be easily verified and
we complete the proof of (\ref{eq:alphabeta}).

To prove that $\alpha \circ \beta=I$, it suffices to show that
\begin{equation}\label{eq:betaalpha}
\alpha (\beta(v))=v
\end{equation}
with $v$ avoiding $\{31425, 32415,31524,32514\}$.
The cases when $\lrmax(v)=1$, $\rlmax(v)=1$ or $a_h=h$ can be check easily. We now explore the non-trivial cases
when $\lrmax(v)>1$, $\rlmax(v)>1$ and $a_h>h$.
Again, we proceed by induction on $n$.

If $v_{a_{h-1}}< v_{b_2}$, let $\tilde v=v_1 \cdots v_{a_{h}-1} v_{a_{h}+1}\cdots v_n$. We consider the following three cases.
\begin{itemize}
  \item When $v$ is  of type A-1, namely $a_{h-1}+1=a_h$, $\beta(v)$ is obtained by inserting
      $v_{a_h}$ just after $v_{a_{h-1}}$ in $\beta(\tilde v)$.  It follows from item 1 in Proposition~\ref{prop:beta} and Observation~\ref{prop:relativeorderB} that $\beta(v)$ is
      a word of type I-1 in $\W_n(31245,32145,31254,32154)$. Consequently, inserting $v_{a_h}$ after $v_{a_{h-1}}$ in $\alpha(\beta(\tilde v))$, we obtain $\alpha(\beta(v))$.
      As $\alpha(\beta(\tilde v))=\tilde v$ by induction, we see that $\alpha(\beta(v))=v$ in this case.
  \item  When $v$ is of type A-2 (resp.~A-3),
      $\beta(v)$ is obtained by inserting
      $v_{a_h}$ just after the element closely following
      $v_{a_{h-1}}$ in $\beta(\tilde v)$.
       Based on item 1 in Proposition~\ref{prop:beta}, $v_{a_{h-1}}$ and $v_{a_{h}}$ remain to be  the $(h-1)$-th
      and the $h$-th left to right maxima of $\beta(v)$, respectively. By Observation~\ref{prop:relativeorderB},
      we see that $v_{a_{h-1}}$, $v_{a_{h}}$ and $v_{b_{2}}$
      keep the same relative order in $\beta(v)$.
      Thus, $\beta(v)$ is
      a word of type I-2 (resp.~I-3) in $\W_n(31245,32145,31254,32154)$ in view of item~3 in Proposition~\ref{prop:beta}. Consequently, inserting $v_{a_h}$ before $v_{b_2}$ (resp.~$v_{a_h+1}$)
       in $\alpha(\beta(\tilde v))=\tilde v$, we obtain $\alpha(\beta(v))$ which is equal to $v$.
\end{itemize}

If $v_{a_{h-1}}> v_{b_2}$ and $h>2$, let $\hat v=v_1 \cdots v_{a_{h-1}-1} v_{a_{h-1}+1}\cdots v_n$. We consider the following three cases.
\begin{itemize}
  \item When $v$ is  of type B-1, then $\beta(v)$ is obtained by inserting
      $v_{a_{h-1}}$ just after $v_{a_{h-2}}$ in $\beta(\hat v)$.  It follows that $\beta(v)$ is
      a word of type II-1 in $\W_n(31245,32145,31254,32154)$ through
      item 1 in Proposition~\ref{prop:beta} and Observation~\ref{prop:relativeorderB}. Consequently, inserting $v_{a_{h-1}}$ after $v_{a_{h-2}}$ in $\alpha(\beta(\hat v))$, we obtain $\alpha(\beta(v))$.
      As $\alpha(\beta(\hat v))=\hat v$ by induction, we see that $\alpha(\beta(v))=v$ in this case.
  \item  If $w$ is of type B-2 (resp.~B-3) and $x \neq a_{h-2}$,
      $\beta(v)$ is obtained by inserting
      $v_{a_{h-1}}$ just after $v_{a_{h-2}}$
      in $\beta(\hat v)$. Similarly as that of type B-1, we may prove that $\alpha(\beta(v))=v$.    If $v$ is of type B-2 (resp.~B-3) and $x = a_{h-2}$, then $\beta(v)$ can be obtained by inserting $v_{a_{h-1}}$ just after the element closely following  $v_{a_{h-2}}$ in $\beta(\hat v)$. By items~1 and~3 in Proposition~\ref{prop:beta}, we see that $\beta(v)$
      is a word of type II-2 or II-4 (resp.~II-3) in $\W_n(31245,32145,31254,32154)$. Further, $\alpha(\beta(v))$ is obtained by inserting $v_{a_{h-1}}$
      just before $v_{a_h}$ (resp.~$v_{a_{h-1}+1}$) in $\alpha(\beta(\hat v))$. Since
      $\alpha(\beta(\hat v))=\hat v$ by the induction hypothesis, we have $\alpha(\beta(v))=v$.
\end{itemize}

The verification of~\eqref{eq:betaalpha} when $v_{a_{h-1}}> v_{b_2}$ and $h=2$ is plain, which
completes the proof of this proposition.
\end{proof}

Combining Proposition~\ref{prop:alpha}, Lemma~\ref{vrmax:stat} and Proposition~\ref{prop:bij}, we finish  the proof of Theorem~\ref{thm:fiveplus}.

In the following, we give an example of the maps $\alpha$ and $\beta$. During the insertion procedure, we need only care about the inserted element, the ``landmark'' element (i.e., the element before/after which we insert) and their relative order in positions.
\begin{example}\label{ex:five}
Assume that $\pi \in \SS_{23}(31245,32145,31254,32154)$ and
\[\pi=23 \, 1 \, 3\, 10\, 18\, 2\, 22\, 21\, 19\, 16\, 14\, 20\, 15\, 11\, 17\, 12\, 9\, 6\, 13\, 7\, 8\, 4\, 5.\]
Then $l_1=1$ with $s=1$ and
$(r_1,r_2,r_3,r_4,r_5,r_6,r_7,r_8)=(1,7,8,12,15,19,21,23)$ with $t=8$.

By the construction of $\alpha$,   we  deduce that
$\alpha(\pi)=23\, \alpha(\pi_2 \cdots \pi_{23})$.
 Based on $\alpha(1\,3\,2\,4\,5)=\varphi(1\,3\,2\,4)\,5=1\,3\,2\,4\,5$,  $\alpha(\pi_2 \cdots \pi_{23})$ can be obtained by
induction with the length of the word decreasing one by one  as follows.
We write, as an example, ``$22$ before $21$'' instead of ``$22$ is inserted before $21$''
for convenience.
\begin{align*}
  &\text{$22$  before $21$} \rightarrow \text{$21$  before $19$} \rightarrow \text{$19$ before $20$} \rightarrow \text{$18$ after  $10$}\rightarrow \text{$20$ before $17$}\rightarrow \text{$16$ before $14$}\\[2pt]
 & \rightarrow\text{$15$  after $14$} \rightarrow \text{$14$  before $11$} \rightarrow \text{$17$ after $11$} \rightarrow \text{$12$ after  $11$}\rightarrow \text{$11$ before $13$}\rightarrow \text{$10$ after $3$}\rightarrow\\[2pt]
 & \text{$9$  before $6$} \rightarrow \text{$13$  before $7$} \rightarrow \text{$7$ after $6$} \rightarrow \text{$6$ before  $8$}\rightarrow \text{$8$ before $4$}.
  \end{align*}
  By inserting elements in $1\,3\,2\,4\,5$ in reverse order as given above, we obtain that
  \[\alpha(\pi)=23\, 1 \, 3 \, 10\, 18\,2\,9\,6\, 16\,14\, 15\, 11\, 22\, 21\,
  19\, 20\, 17\, 12\, 13\,7\,8\, 4\, 5.\]

On the other hand, assume that $p \in \SS_n(31425,32415,31524,32514)$ and
  \[p =23\, 1 \, 3 \, 10\, 18\,2\,9\,6\, 16\,14\, 15\, 11\, 22\, 21\,
  19\, 20\, 17\, 12\, 13\,7\,8\, 4\, 5.\]
Based on the construction of $\beta$, $\beta(p)=23 \, \beta(p_2 \cdots p_{23})$. Similarly, $\beta(\pi_2 \cdots \pi_{23})$ can be obtained via inserting elements  in  $\beta(1\,3\,2\,4\,5)=\psi(1\,3\,2\,4)5=1\,3\,2\,4\,5$ in the reverse order given as follows.
\begin{align*}
  &\text{$22$  after $2$} \rightarrow \text{$21$  after $2$} \rightarrow \text{$19$ after $2$} \rightarrow \text{$18$ after  $10$}\rightarrow \text{$20$ after $14$}\rightarrow \text{$16$ after $2$}\rightarrow\\[2pt]
 & \text{$15$  after $14$} \rightarrow \text{$14$  after $2$} \rightarrow \text{$17$ after $11$} \rightarrow \text{$12$ after  $11$}\rightarrow \text{$11$ after $2$}\rightarrow \text{$10$ after $3$}\rightarrow\\[2pt]
 & \text{$9$  after $2$} \rightarrow \text{$13$  after $6$} \rightarrow \text{$7$ after $6$} \rightarrow \text{$6$ after  $2$}\rightarrow \text{$8$ after $2$}.
  \end{align*}
  It follows that  $\beta(p)=23 \, 1 \, 3\, 10\, 18\, 2\, 22\, 21\, 19\, 16\, 14\, 20\, 15\, 11\, 17\, 12\, 9\, 6\, 13\, 7\, 8\, 4\, 5.$
\end{example}

\section{Revisiting $(201,210)$-avoiding inversion sequences} 
\label{sec:4}

The {\em inversion sequences} of length $n$,
$$
\I_n:=\{(e_1,e_2,\ldots,e_n)\in\N^n: 0\leq e_i<i\},
$$
serve as various kinds of codings for $\SS_n$. By a coding of $\SS_n$, we mean a bijection from $\SS_n$ to $\I_n$.
For example, the famous {\em Lehmer code}   $\Theta: \SS_n\rightarrow\I_n$ is defined  as
$$
\Theta(\pi)=(e_1,e_2,\ldots,e_n),\quad\text{where $e_i:=|\{j: j<i\text{ and }\pi_j>\pi_i\}|$}
$$
for each $\pi\in\SS_n$.
The interplay between inversion sequences and permutations possess a number of unexpected  applications in studying patterns and statistics~\cite{AE,kl,ms,CL,CN,HZ,sv}.

The study of enumerations and bijections for inversion sequences avoiding multiple patterns of length 3 was initiated by Martinez and Savage~\cite{ms} and continued  by many other researchers (see~\cite{bbgr,Lin,CL} and the references therein). In this section, we revisit $(201,210)$-avoiding inversion sequences and show how they can help to prove a refinement of Gao--Kitaev's conjecture and compute the generating function for the sequence A212198  that counts the three classes of pattern avoiding permutations in concern.

\subsection{On Martinez--Savage's coding $\phi$ and a refinement of Gao--Kitaev's conjecture}
\label{sec:phi}
The permutation code $\phi:\SS_n\rightarrow\I_n$ introduced by Martinez and Savage~\cite{ms} will be used to obtain  a refinement of Gao--Kitaev's conjecture. For $\pi\in\SS_n$, the inversion sequence $\phi(\pi)=(e_1,e_2,\ldots,e_n)\in\I_n$ is defined by first setting $e_n=\pi_n-1$ and then for $i$ from $n-1$ to $1$ do
\begin{itemize}
\item if $\pi_i\leq i$, then set $e_i=\pi_i-1$;
\item otherwise, $\pi_i>i$ and if $\pi_i$ is the $k$-th largest element  in $\{\pi_1,\pi_2,\ldots,\pi_i\}$, then set $e_i$ to be the $k$-th smallest in  $\{e_j: i<j\leq n\}$.
\end{itemize}
For example, if $\pi=582937416\in\SS_9$, then $\phi(\pi)=(0,0,1,0,2,5,3,0,5)\in\I_n$. Note that whenever $\pi_i>i$ and $\pi_i$ is the $k$-th largest  in $\{\pi_1,\pi_2,\ldots,\pi_i\}$, there are at least $k$ letters not greater than $i$ occurring after $\pi_i$ in $\pi$, which forces $e_i<i$ and so $\phi(\pi)$ is really  an inversion sequence.

Let us introduce some statistics on permutations and inversion sequences.
For each permutation $\pi\in\SS_n$,
\begin{itemize}
\item  $\lmaxz(\pi)$ is one plus the number of left-to-right maxima of $\pi$ appear before the letter $1$ in $\pi$;
\item   $\exc(\pi):=|\{i\in[n-1]:\pi_i>i\}$|, the number of {\em excedances} of $\pi$.
\end{itemize}
For each inversion sequence $e\in\I_n$,
\begin{itemize}
\item $\dist(e):=|\{e_1,e_2,\ldots,e_n\}\setminus\{0\}|$, the number of {\em distinct positive entries} of $e$;
\item  $\rep(e):=n-1-\dist(e)$, the number of times that entries of $e$ are repeated;
\item $\rlmin(e):=|\{i\in[n]: e_i<e_j \text{ for all $j>i$\,}\}|$, the number of {\em right-to-left minima} of $e$;
\item $\zero(e):=|\{i\in[n]: e_i=0\}|$, the number of {\em zero entries} in $e$.
\end{itemize}
We observe the following property of $\phi$.
\begin{lemma}\label{lem:rlmin}
The coding $\phi:\SS_n\rightarrow\I_n$ transforms the triple of statistics $(\exc,\rlmin,\lmaxz)$ to $(\rep,\rlmin,\zero)$.
\end{lemma}
\begin{proof}
Let $\pi\in\SS_n$ and let $e=\phi(\pi)$. The fact that $\exc(\pi)=\rep(e)$ is obvious from the construction of $\phi$, which was known in~\cite{ms}. If $\pi_i$ is a right-to-left minimum of $\pi$, i.e., $\pi_i<\pi_j$ for all $j\geq i$, then $\pi_i\geq i$. Thus, $e_i=i-1$ is a right-to-left minimum of $e$. This proves $\rlmin(\pi)=\rlmin(e)$.

It remains to show that $\lmaxz(\pi)=\zero(e)$. Suppose that $\pi_k=1$ for some $k$. Then $e_k=0$ is the rightmost zero entry of $e$.
Moreover, if $k>1$ and $\pi_i$ is a left-to-right maximum  of $\pi$ appears before the letter $1$, then $\pi_i>i$ and $\pi_i$ is largest in $\{\pi_1,\pi_2,\ldots,\pi_i\}$, which forces $e_i=0$ by the definition of $\phi$. This proves $\lmaxz(\pi)=\zero(e)$.
\end{proof}

  Martinez and Savage~\cite[Theorem~56]{ms} showed that the coding $\phi$ restricts to a bijection between $\SS_n(45312,45321,54312,54321)$ and $\I_n(201,210)$. In view of Lemma~\ref{lem:rlmin}, we have

\begin{proposition}
\label{prop:ChenL}
For $n\geq1$,
\begin{equation}\label{eq:phi}
\sum_{\pi\in\SS_n(45312,45321,54312,54321)}t^{\exc(\pi)}p^{\lrmin(\pi)}q^{\lmaxz(\pi)}=\sum_{e\in\I_n(201,210)}t^{\rep(e)}p^{\rlmin(e)}q^{\zero(e)}.
\end{equation}
\end{proposition}
On the other hand, our work in~\cite[Proposition~3.7]{CL} proves  that  Baril and Vajnovszki's $b$-code~\cite{BV} restricts to a bijection between $\SS_n(45312,45321,54312,54321)$ and $\I_n(201,210)$ and hence
\begin{equation}\label{eq:bcode}
\sum_{\pi\in\SS_n(24135,24153,42135,42153)}t^{\ides(\pi)}p^{\rlmax(\pi)}q^{\lrmax(\pi)}=\sum_{e\in\I_n(201,210)}t^{\dist(e)}p^{\rlmin(e)}q^{\zero(e)}.
\end{equation}
Combining~\eqref{eq:phi} and~\eqref{eq:bcode} gives

\begin{proposition}\label{prop:bphi}
For $n\geq1$,
\begin{equation*}
\sum_{\pi\in\SS_n(45312,45321,54312,54321)}t^{\exc(\pi)}p^{\lrmin(\pi)}q^{\lmaxz(\pi)}=\sum_{\pi\in\SS_n(24135,24153,42135,42153)}t^{\iasc(\pi)}p^{\rlmax(\pi)}q^{\lrmax(\pi)},
\end{equation*}
where $\iasc(\pi):=n-1-\ides(\pi)$ is the number of ascents of $\pi^{-1}$.
\end{proposition}

Since the inverse $\pi\mapsto\pi^{-1}$ sets up a bijection between $\SS_n(24135,24153,42135,42153)$ and
$\SS_n(31425,32415,31524,32514)$ and transforms the triple of statistics $(\iasc,\rlmax,\lrmax)$ to $(\asc,\rlmax,\rlmin)$, we have
$$
\sum_{\pi\in\SS_n(24135,24153,42135,42153)}t^{\iasc(\pi)}p^{\rlmax(\pi)}q^{\lrmax(\pi)}=\sum_{\pi\in\SS_n(31425,32415,31524,32514)}t^{\asc(\pi)}p^{\rlmax(\pi)}q^{\rlmin(\pi)}.
$$
The following refinement of Gao--Kitaev's conjecture then follows from Proposition~\ref{prop:bphi} and  Theorem~\ref{thm:fiveplus}.
\begin{proposition}[A refinement of Gao--Kitaev's conjecture]
\label{refine:GK}
For $n\geq1$,
$$
\sum_{\pi\in\SS_n(45312,45321,54312,54321)}p^{\lrmin(\pi)}=\sum_{\pi \in \SS_n(31245,32145,31254,32154)} p^{\rlmax(\pi)}.
$$
\end{proposition}

\subsection{A succession rule for $(201,210)$-avoiding inversion sequences}
It has been a widely used method to prove that two pattern-avoiding classes have the same cardinality by showing they obey the same succession rule; see~\cite{bbgr,FL2021,kl,ms}. Although we can show that $(201,210)$-avoiding inversion sequences do obey a simple succession rule, we failed  to find any succession rule for $(31245,32145,31254,32154)$-avoiding permutations.

For each $e\in\I_n(201,210)$, define the {\em parameters} $(p,q)$ of $e$, where
$$
p=|\{k>e_n: (e_1,e_2,\ldots,e_n,k)\in\I_{n+1}(201,210)\}|
$$
and
$$
q=|\{k\leq e_n: (e_1,e_2,\ldots,e_n,k)\in\I_{n+1}(201,210)\}|.
$$
For example, if $e=(0,1,0,2,4,2,5)\in\I_7(201,210)$, then the parameters of $e$  is $(2,3)$.
We have the following succession rule for $(201,210)$-avoiding inversion sequences.

\begin{lemma}\label{lem:201}
Suppose that $e\in\I_n(201,210)$ has parameters $(p,q)$. Exactly $p+q$  inversion sequences in $\I_{n+1}(201,210)$ when removing their last entries will become $e$, and their  parameters are respectively:
\begin{align*}
&(p,q+1), (p-1,q+2), \ldots, (1,q+p), \\
&(p+1,q), \underbrace{(p+2,1),\ldots, (p+2,1)}_{q-1}.
\end{align*}
\end{lemma}

\begin{proof}
Suppose that $k_1,k_2, \ldots,  k_p$   with $ k_1<k_2<\ldots<k_p=n$ are the integers such that   $(e_1,e_2,\ldots,e_n,k_i)\in \I_{n+1}(201,210)$ for $1\leq i\leq p$.
 Then the inversion sequence   $(e_1,e_2,\ldots,e_n,k_i) $  has  the parameters $(p+1-i,q+i)$ for $1\leq i\leq p$.  On the other hand, if  $e_n\geq l_1>l_2>\ldots>l_q$ are the integers such that   $(e_1,e_2,\ldots,e_n,l_i)\in \I_{n+1}(201,210)$ for $1\leq i\leq q$,
 then  the inversion sequence   $(e_1,e_2,\ldots,e_n,l_i) $  has  the parameters
 $$
 \begin{cases}
 (p+1,q),\qquad &\text{if $i=1$}\\
 (p+2,1), &\text{otherwise}.
 \end{cases}
 $$
This completes the proof of the lemma.
\end{proof}

\begin{remark}
It would be interesting to show that $(31245,32145,31254,32154)$-avoiding permutations admit certain same  succession rule  as that of $(201,210)$-avoiding inversion sequences. This will lead to another proof of Gao--Kitaev's conjecture.
\end{remark}

Let $F(u,v;t)=F(u,v):=\sum_{p,q\geq1} f_{p,q}(t)u^pv^q$, where $f_{p,q}(t)$ is the size generating function of the $(201,210)$-avoiding inversion  sequences with parameters $(p,q)$. We can turn the succession rule in Lemma~\ref{lem:201} into functional equation as follows.

\begin{proposition}
We have the following functional equation for $A(u,v)$:
\begin{equation}\label{fun:201}
F(u,v)=tuv+t\biggl(\frac{uF(u,v)-vF(v,v)}{1-v/u}+u^2v\biggl(\left.\dfrac{\partial F(u,v)}{\partial v}\right|_{v=1}-F(u,1)\biggr)\biggr).
\end{equation}
Equivalently, if we write $F(u,v)=\sum_{n\geq1} f_n(u,v)t^n$, then $f_1(u,v)=uv$ and for $n\geq2$,
\begin{equation}
\label{rec:201}
f_n(u,v)=\frac{uf_{n-1}(u,v)-vf_{n-1}(v,v)}{1-v/u}+u^2v\biggl(\left.\dfrac{\partial f_{n-1}(u,v)}{\partial v}\right|_{v=1}-f_{n-1}(u,1)\biggr).
\end{equation}
\end{proposition}
\begin{proof}
We construct a generating tree for $(201,210)$-avoiding inversion sequences by representing each element as its parameters like this: the root is $(1,1)$ and the children of a vertex labelled $(p,q)$ are those that generated according to the succession rule in Lemma~\ref{lem:201}. Then the vertices in the $n$th level of this generating tree corresponding to the parameters of the sequences in $\I_n(201,210)$. In this generating tree, every vertex other than the root $(1,1)$ can be generated by a unique parent. Thus, we have
\begin{align*}
F(u,v)&=tuv+ t\sum_{p,q\geq1}f_{p,q}(t)\biggl(\sum_{i=1}^{p}u^{p+1-i}v^{q+i}+u^{p+1}v^q+(q-1)u^{p+2}v\biggr)\\
&=tuv+t\sum_{p,q\geq1}f_{p,q}(t)\biggl(\frac{u^{p+1}v^{q}-v^{p+q+1}}{1-v/u}+(q-1)u^{p+2}v\biggr)\\
&=tuv+t\biggl(\frac{uF(u,v)-vF(v,v)}{1-v/u}+u^2v\biggl(\left.\dfrac{\partial F(u,v)}{\partial v}\right|_{v=1}-F(u,1)\biggr)\biggr),
\end{align*}
which gives~\eqref{fun:201}.
\end{proof}

Although we could not solve~\eqref{fun:201}, recursion~\eqref{rec:201} can be applied to compute $f_n(u,v)$ and thus
$$
f_n(1,1)=|\I_n(201,210)|=|\SS_n(31245,32145,31254,32154)|.
$$
However, we will show in next section that the generating function for $(201,210)$-avoiding inversion sequences satisfies an algebraic equation of degree $2$, to our surprise.

\subsection{The generating function for $|\I_n(201,210)|$ is algebraic}
For $e\in\I_n$,  an entry $e_i$ of $e$ is {\em saturated} if $e_i=i-1$.
Introduce $A(t,q):=\sum_{n\geq1}t^n\sum_{e\in\I_n(201,210)}q^{\satu(e)}$, where $\satu(e)$ denotes  the number of saturated entries in $e$.  Let
$$
A(t):=A(t,1)=t+2t^2+6t^3+24t^4+116t^5+632t^6+3720t^7+\cdots
$$
be the generating function for $|\I_n(201,210)|$.

\begin{theorem}
The generating function $A(t)$ for $(201,210)$-avoiding inversion sequences satisfies the algebraic equation
\begin{equation}\label{alg:gen}
(2t^2-2t+1)A^2+(4t^2-3t)A+2t^2=0,
\end{equation}
whose formal power series solution is
$$
A(t)=\frac{3t-4t^2-t\sqrt{1-8t}}{4t^2-4t+2}.
$$
\end{theorem}
\begin{proof}
Let $e=(e_1,e_2,\ldots,e_n)\in\I_n(201,210)$. Let $e_{\ell}$ be the rightmost  saturated entry of $e$, that is $\ell=\max\{i\in[n]: e_i=i-1\}$. We need to consider two cases:
\begin{itemize}
\item[(a)] If $\ell=n$, then
$$
\satu(e)=\satu(e_1,e_2,\ldots,e_{n-1})+1.
$$
\item[(b)] Otherwise, $1\leq \ell<n$. We need further to distinguish two cases:
\begin{itemize}
\item[(b1)] If $e_{\ell+1}=e_{\ell}=\ell-1$, then $(e_1,e_2,\ldots,e_{\ell},e_{\ell+2},\ldots,e_n)\in\I_{n-1}(201,210)$ and
$$
\satu(e)=\satu(e_1,e_2,\ldots,e_{\ell}).
$$
\item[(b2)] Otherwise, as $e_{\ell}$ is the rightmost saturated entry, $e_{\ell+1}<e_{\ell}=\ell-1$. In this case, since $e$ is $(201,210)$-avoiding and $e_{\ell}$ is the rightmost saturated  entry of $e$, either $e_i=e_{\ell+1}$ or $e_i\geq e_{\ell}$ for $i=\ell+2,\ldots,n$. It is straightforward to show that $e$ can be decomposed into
$$\tilde{e}:=(e_1,e_2,\ldots,e_{\ell-1},e_{\ell+1})\in\tilde{\I}_{\ell}(201,210)$$
and
$$\bar{e}:=(0,\bar e_{\ell+1},\bar e_{\ell+2},\ldots,\bar{e}_n)\in\bar{\I}_{n+1-\ell}(201,210),$$
where
$$
\bar e_i:=
\begin{cases}
0&\quad\text{if $e_i=e_{\ell+1}$,}\\
e_i-e_{\ell}+1&\quad\text{if $e_i\geq e_{\ell}$,}
\end{cases}
$$
 for $\ell+1\leq i\leq n$.
Here
$$\tilde{\I}_{\ell}(201,210):=\{e\in\I_{\ell}(201,210): e_{\ell}\text{ is not saturated}\}$$
and
$$
\bar{\I}_{\ell}(201,210):=\{e\in\I_{\ell}(201,210): e_{2}=0\}.
$$
This decomposition is reversible and satisfies the property
$$
\satu(e)=\satu(e_1,e_2,\ldots,e_{\ell-1},e_{\ell+1})+1.
$$
\end{itemize}
\end{itemize}

Clearly, counting the inversion sequences from case (a) (by length and the number of saturated entries) gives
$$
tq+tqA(t,q).
$$
Notice that
 $\sum_{n\geq2}t^n\sum_{e\in\widetilde{\I}_n(201,210)}q^{\satu(e)}=A(t,q)-tq-tqA(t,q)$ and there is an obvious one-to-one correspondence between $\bar{\I}_{\ell}(201,210)$ and $\I_{\ell}(201,210)\setminus\bar{\I}_{\ell}(201,210)$ for $\ell\geq2$. Thus, the generating function for the inversion sequences in case (b2) is
$$
q(A(t,q)-tq-tqA(t,q))\frac{A(t)-t}{2t}.
$$
Finally, as each inversion sequence in case (b1) is obtained from some $(201,210)$-avoiding inversion sequence by inserting one copy of a saturated entry immediately to the right of this saturated entry, if written $A(t,q)$ as $\sum_{k,n\geq1}a_{n,k}t^nq^k$, then the generating function for case  (b1) is
$$
t\sum_{k,n\geq1}a_{n,k}t^n(q+q^2+\cdots+q^k)=t\sum_{k,n\geq1}a_{n,k}t^n\biggl(\frac{q-q^{k+1}}{1-q}\biggr)=\frac{tq}{1-q}(A(t)-A(t,q)).
$$
Summing over all the above cases yields the following functional equation for $A(t,q)$:
\begin{equation*}
A(t,q)=tq+tqA(t,q)+q(A(t,q)-tq-tqA(t,q))\frac{A(t)-t}{2t}+\frac{tq}{1-q}(A(t)-A(t,q)),
\end{equation*}
which is equivalent to
\begin{equation}\label{ker:1}
\biggl(1+\frac{tq^2}{1-q}-\frac{q(1-tq)(A(t)-t)}{2t}\biggr)A(t,q)=tq+\frac{q^2(A(t)-t)}{2}+\frac{tqA(t)}{1-q}.
\end{equation}

We apply the kernel method to solve~\eqref{ker:1} and set the coefficient
$$1+\frac{tq^2}{1-q}-\frac{q(1-tq)(A(t)-t)}{2t}$$ of $A(t,q)$ to be zero, then the right-hand side of~\eqref{ker:1} vanishes. Therefore, $A(t)$ satisfies the system of equations
\begin{equation}\label{sys:ker1}
\begin{cases}
\,\,1+\frac{tq^2}{1-q}-\frac{q(1-tq)(A(t)-t)}{2t}=0,\\
\,\,tq+\frac{q^2(A(t)-t)}{2}+\frac{tqA(t)}{1-q}=0.
\end{cases}
\end{equation}
Eliminating $A(t)$ yields the algebraic equation for $q=q(t)$:
\begin{equation}\label{eq:q}
(t+1)q^2-3q+2=0.
\end{equation}
Equivalently, we have
$q^2=\frac{3q-2}{t+1}$.
Involving this equality, the second equation in~\eqref{sys:ker1} gives
$$
q=\frac{2((t^2+t-1)A+t^2+2t)}{(t-2)A+t^2+4t}.
$$
Substituting into~\eqref{eq:q} results in~\eqref{alg:gen} after simplification.
\end{proof}

By accident, we find that Albert, Linton and  Ru\v{s}kuc~\cite[Page~20]{ALR} have proved that the generating function for the class of $(41325, 51324, 42315, 52314)$-avoiding permutations  shares the same algebraic equation~\eqref{alg:gen} as $A(t)$.  Thus, together with Theorem~\ref{thm:fiveplus}, Propositions~\ref{prop:ChenL} and~\ref{refine:GK} we get the following five interpretations for A212198 in the OEIS~\cite{oeis}.

\begin{corollary}\label{cor:qua}
The following five pattern-avoiding classes
\begin{align*}
&\SS_n(45312,45321,54312,54321),\\
&\SS_n(31245,32145,31254,32154),\\
&\SS_n(31425,32415,31524,32514),\\
&\I_n(201,210)\quad\text{and}\\
&\SS_n(41325, 51324, 42315, 52314)
\end{align*}
all interpret the integer sequence A212198.
\end{corollary}

Note that the four classes of pattern-avoiding permutations above are not in bijection with each other under the fundamental symmetry operations: the inverse, the complement and the reversal of permutations.

\section{Final remarks, open problems}

Each quadruple of patterns in the four classes of permutations in Corollary~\ref{cor:qua} posses the same phenomenon: fix a letter and then exchange  respectively two pairs of the other letters. For instance, the quadruple $(31245,32145,31254,32154)$  can be obtained by  first fixing letter $3$ in  position $1$ and then exchanging the pairs of letters  $(1,2)$ (in positions $2$ and $3$) and $(4,5)$ (in positions $4$ and $5$) to get the four patterns. Using this principle, we find  the following conjectured 13 classes that are enumerated by the integer sequence A212198. 

\begin{conj}\label{conj:CL}
The following 13 classes are enumerated by  A212198:
\begin{align*}
&\SS_n(45312,45321,54312,54321),\\
&\SS_n(31245,32145,31254,32154),\\
&\SS_n(31425,32415,31524,32514),\\
&\SS_n(41325, 51324, 42315, 52314),\\
&\SS_n(13425,23415,13524,23514),\\
&\SS_n(13452,23451,13542,23541),\\
&\SS_n(24513,25413,24531,25431),\\
&\SS_n(13245,23145,13254,23154),\\
&\SS_n(32415,34215,32451,34251),\\
&\SS_n(21345,23145,23154,21354),\\
&\SS_n(24135,25134,25314,24315),\\
&\SS_n(42513,52413,42531,52431),\\
&\SS_n(42135,52134,52314,42315).
\end{align*}
\end{conj}
The first four classes of pattern-avoiding permutations above have been proved by Corollary~\ref{cor:qua}. Recently,  Pantone~\cite{Pan} has informed us that he managed to automatically find specifications for the above $13$ classes using their algorithmic framework for enumeration~\cite{Jay} and so Conjecture~\ref{conj:CL} was confirmed in full via generating functions. The details will appear in the website of PermPAL~\cite{PermPAL}. 

The main achievement of this paper is the construction of a bijection between 
$$\SS_n(31245,32145,31254,32154)\quad\text{and}\quad\SS_n(31425,32415,31524,32514)
$$
preserving  the quintuple of set-valued statistics $(\Ides,\Vlrmax,\Vlrmin,\Vrlmax,\Iar)$, which leads to a refinement of Gao--Kitaev's conjecture; see Theorem~\ref{thm:fiveplus} and Proposition~\ref{refine:GK}. It would be interesting to explore refinements of Conjecture~\ref{conj:CL} using classical permutation statistics from the bijective aspect.

\section*{Acknowledgement}
We thank J. Pantone for drawing our attention to their powerful algorithmic framework for enumeration~\cite{Jay} that could verify  Conjecture~\ref{conj:CL} completely  by generating functions.  This work was supported by
the National Science Foundation of China grants 12271301 and 11701420, and  the project of Qilu Young Scholars of Shandong University.

\end{document}